\documentclass[12pt,fleqn]{article}
\usepackage[english]{babel}
\usepackage{makeidx}
\usepackage{latexsym,amsfonts,amssymb,amsmath,longtable,amsthm}
\input amssym.def
\usepackage{latexsym}
\usepackage[all]{xy}
\usepackage{amsfonts}
\usepackage{amsmath}
\usepackage{mathrsfs}
\usepackage{color}
\usepackage{ulem}
\usepackage[table]{xcolor}
\usepackage[unicode,breaklinks=true,colorlinks=true]{hyperref}

\catcode`@=11
\newbox\tr@tto
\setbox\tr@tto=\hbox{{\count0=0\dimen0=-
,9pt\dimen1=1,1pt\loop\ifnum\count0<11 \advance \count0 by1 \vrule
width.51pt height\dimen1
     depth\dimen0\kern-0.17pt\advance\dimen0 by-
0.05pt\advance\dimen1
     by0.1pt\repeat \loop\ifnum\count0<21\advance \count0 by1
\vrule
     width.6pt height\dimen1 depth\dimen0\kern-0.2pt
\advance\dimen0
     by-0.1pt\advance\dimen1 by 0.05pt\repeat}}
\def\medint{\displaystyle\copy\tr@tto\kern-10.4pt\int}
\catcode`@=12

\def\Xint#1{\mathchoice
   {\XXint\displaystyle\textstyle{#1}}%
   {\XXint\textstyle\scriptstyle{#1}}%
   {\XXint\scriptstyle\scriptscriptstyle{#1}}%
   {\XXint\scriptscriptstyle\scriptscriptstyle{#1}}%
   \!\int}
\def\XXint#1#2#3{{\setbox0=\hbox{$#1{#2#3}{\int}$}
     \vcenter{\hbox{$#2#3$}}\kern-.5\wd0}}

\def\dashint{\Xint-}

\newcommand{\R}{{\mathbb R}}
\newcommand{\tF}{{\widetilde F}}

\newcommand{\cF}{{\mathcal F}}
\newcommand{\ttF}{{\widetilde G}}

\newcommand{\N}{{\mathbb N}}
\newcommand{\Z}{{\mathbb Z}}

\newcommand{\e}{{\varepsilon}}
\newcommand{\al}{{\alpha}}
\newcommand{\Le}{{\mathscr L}}
\newcommand{\E}{\mathop{\rm Ent}}
\newcommand{\M}{{\mathcal  M}}
\newcommand{\Be}{{\mathcal  B}}
\newcommand{\Me}{{\mathscr  M}}
\renewcommand{\H}{{\mathcal H}}
\newcommand{\rank}{{\mathrm{rank\,}}}

\newcommand{\bg}{{\bar g}}

\newcommand{\p}{{p_\circ}}
\renewcommand{\b}{{q_\circ}}

\newcommand{\loc}{{\rm loc}}

\newcommand{\diam}{\mathop{\rm diam}}

\newcommand{\LL}{\mathrm{L}}
\newcommand{\WW}{\mathrm{W}}

\newcommand{\CC}{\mathrm{C}}

\newtheorem{ttt}{Theorem}[section]

\newtheorem{lem}{Lemma}[section]
\newtheorem{cor}{Corollary}[section]
\theoremstyle{remark}
\newtheorem{rem}{Remark}[section]

\numberwithin{equation}{section}

\newcommand{\dd}{{\rm d}}

\title{A~BRIDGE BETWEEN DUBOVITSKI\u{I}--FEDERER THEOREMS AND THE COAREA FORMULA}
\author{Piotr Haj\l{}asz, \, Mikhail V.~Korobkov,  \, and \, Jan Kristensen }

\begin{document}

\maketitle

\begin{abstract}
The Morse--Sard theorem requires that a mapping $v \colon
\R^n\to\R^m$ is of class $\CC^k$, $k > \max(n-m, 0)$. In 1957
Dubovitski\u{\i} generalized this result by proving that almost
all level sets for a $\CC^k$ mapping have $\H^s$-negligible
intersection with its critical set, where $s=\max(n-m-k+1,0)$.
Here the critical set, or $m$-critical set is defined as $Z_{v,m}
= \{ x \in \R^n : \, \rank \nabla v(x) < m \}$. Another
generalization was obtained independently by Dubovitski\u{\i} and
Federer in 1966, namely for $\CC^k$ mappings $v \colon \R^n \to
\R^d$ and integers $m \leq d$ they proved that the set of
$m$--critical values $v(Z_{v,m})$ is $\H^{\b}$-negligible for $\b
= m-1+\tfrac{n-m+1}{k}$. They also established the sharpness of
these results within the $\CC^k$ category.

Here we prove that Dubovitski\u{\i}'s theorem can be generalized
to the case of continuous mappings of the Sobolev--Lorentz class
$\WW^{k}_{p,1}(\R^n,\R^d )$, $p=\frac{n}k$ (this is the minimal
integrability assumption that guarantees the continuity of
mappings). In this situation the mappings need not to be
everywhere differentiable and in order to handle the set of
nondifferentiability points, we establish for such mappings an
analog of the Luzin $N$--property with respect to lower
dimensional Hausdorff content. Finally, we formulate and prove
a \textit{bridge theorem} that includes all the above results as
particular cases. This  result is new also for
smooth mappings, but is presented here in the general Sobolev
context. The proofs of the results are based on our previous joint papers
with J.~Bourgain (2013, 2015).
\medskip

 {\bf Note, that in this paper some result concerning the Coarea formula was not formulated accurately. Now we put an Addendum consisting of three parts: first, we describe the accurate formulation of this result, then we give some historical remarks, and finally its relation to other results of the paper.}

\end{abstract}


\section{Introduction}\label{Introd}
The Morse--Sard theorem in its classical form states that the
image of the set of critical points of a $\CC^{n-m+1}$ smooth
mapping $v\colon \R^n \to \R^m$ has zero Lebesgue measure in
$\R^m$. More precisely, assuming that $n \geq m$, the set of
critical points for $v$ is $Z_v=\{x\in\R^n\,:\,\rank \nabla
v(x)<m\}$ and the conclusion is that
\begin{equation}\label{classical}
\Le^m(v(Z_v))=0.
\end{equation}
The theorem was proved by Morse~\cite{Mo} in the case $m=1$ and
subsequently by Sard~\cite{S} in the general vector--valued case.
The celebrated results of Whitney~\cite{Wh} show that the
$\CC^{n-m+1}$ smoothness assumption on the mapping $v$ is sharp.
However, the following result gives valuable information also for
less smooth mappings. \vspace{2mm}

\noindent {\bf Theorem A
(Dubovitski\u{\i}~1957~\cite{Du})}.\label{AA} {\sl Let $n,m,k \in
\N$, and let
 $v \colon \R^n\to\R^m$ be a
$\CC^k$--smooth mapping. Put $s=n-m-k+1$. Then
\begin{equation}\label{dub1}
\H^s (Z_{v}\cap v^{-1}(y))=0\qquad\mbox{ for a.a. }y\in\R^m,
\end{equation}
where  $\H^s$ denotes the $s$--dimensional  Hausdorff measure and
$Z_{v}$ is the set of critical points of~$v$.} \vspace{2mm}

Here and in the following we interpret $\H^\beta$ as the counting
measure when $\beta\le0$. Thus for $k\ge n-m+1$ we have $s \le0$,
and $\H^s$ in (\ref{dub1}) becomes simply the counting measure, so
the Dubovitski\u{\i} theorem contains the Morse--Sard theorem as
particular case\footnote{It is interesting to note that because of
the isolation of the former Soviet Union this first
Dubovitski\u{\i} theorem was almost unknown to Western
mathematicians; another proof was given in the recent
paper~\cite{BHS} covering also some extensions to the case of
H\"{o}lder spaces.}.

A few years later and almost simultaneously,
Dubovitski\u{\i}~\cite{Du2} in 1967 and Federer \cite[Theorem
3.4.3]{Fed} in 1969\footnote{Federer announced~\cite{Fed66} his
result in 1966, this announcement (without any proofs) was sent
on~08.02.1966. For the historical details, Dubovitski\u{\i} sent
his paper~\cite{Du2} (with complete proofs) a month earlier,
on~10.01.1966.} published another important generalization of the
Morse--Sard theorem. \vspace{2mm}

\noindent {\bf Theorem B (Dubovitski\u{\i}--Federer)}.\label{BB}
{\sl For $n,k,d \in \N$ let $m\in\{1,\dots,\min (n,d)\}$ and
 $v \colon \R^n\to\R^d$ be a
$\CC^k$--smooth mapping. Put $q_\circ=m+\frac{s}k$. Then
\begin{equation}\label{fed1}
\H^\b(v(Z_{v,m}))=0.
\end{equation}
where, as above, $s=n-m-k+1$  and $Z_{v,m}$ denotes the set of
$m$--critical points of $v$ defined as
$Z_{v,m}=\{x\in\R^n\,:\,\rank \nabla v(x)<m\}$.} \vspace{2mm}

In view of the wide range of applicability of the above results it
is a natural and compelling problem to decide to what extent they
admit extensions to classes of Sobolev mappings. The first
Morse--Sard result in the Sobolev context that we are aware of is
due to L.~De Pascale~\cite{DeP} (though see also \cite{La}). It
states that (\ref{classical}) holds for mappings $v$ of class
$\WW^{k}_{p,{\rm loc}}(\R^n , \R^m )$ when $k \geq \max (n-m+1,2)$
and $p>n$. Note that by the Sobolev embedding theorem any mapping
on $\R^n$ which is locally of Sobolev class $\WW^{k}_{p}$ for some
$p>n$ is in particular $\CC^{k-1}$, so the critical set $Z_v$ can
be defined as usual.

In the recent paper \cite{HZ} P.~Haj\l{}asz and S.~Zimmerman
proved Theorem~A under the assumption that $v\in \WW^{k}_{p,{\rm
loc}}(\R^n , \R^m )$, $p>n$, which corresponds to that used by
L.~De Pascale~\cite{DeP}.

In view of the existing counter--examples to Morse--Sard type
results in the classical $\CC^k$ context the issue {\it is not}
the value of $k$,~--- that is, how many weak derivatives are
needed. Instead the question is, what are the {\it minimal}
integrability assumptions on the weak derivatives for Morse--Sard
type results to be valid in the Sobolev case. Of course, it is
natural here to restrict attention to continuous mappings, and so
to~require from the considered function spaces that the inclusion
$v\in \WW^k_p(\R^n,\R^d)$ should guarantee at least the continuity
of~$v$ (assuming always that the mappings are precisely
represented). For values $k\in\{1,\dots,n-1\}$ it is well--known
that $v\in \WW^k_p(\R^n,\R^d)$ is continuous for~$p>\frac{n}k$ and
could be discontinuous for $p\le \frac{n}k$. So {\bf the
borderline case} is $p=\p=\frac{n}k$. It is well--known (see for
instance \cite{Maly1,KK3}) that $v\in \WW^k_\p(\R^n,\R^d)$ is
continuous if the derivatives of $k$-th order belong to the
Lorentz space~$\LL_{\p,1}$, we will denote the space of such
mappings by~$\WW^k_{\p,1}(\R^n,\R^d)$. We refer to
section~\ref{prel} for relevant definitions and notation.

In \cite{KK15} it was shown that mappings $v\in \WW^k_{\p
,1}(\R^n,\R^d)$ are differentiable (in the classical
Fr\'{e}chet--Peano sense) at each point outside
some~$\H^{\p}$--negligible set $A_v$ (see
Theorems~\ref{Th_dif}--\ref{Th_dif2}).\footnote{It was also proven
that each point $x\in \R^n\setminus A_v$ is an $\LL_{\p}$-Lebesgue
point for the weak gradient~$\nabla v$. Note that for mappings of
the classical Sobolev space~$\WW^{k}_{\p}(\R^n)$ the corresponding
exceptional set $U$ has small Bessel capacity
$\mathcal{B}_{k-m,p}(U)<\varepsilon$, and, respectively, the
gradients $\nabla^m v$ are well-defined in $\R^n$ except for some
exceptional set of zero Bessel capacity~$\mathcal{B}_{k-m,p}$
(see, e.g., Chapter~3 in \cite{Ziem} or \cite{BHS}\,).} Thus we
define for integers $m \leq \min \{ n,d \}$ the
\textbf{$m$--critical set} as
\begin{equation}\label{criset}
Z_{v,m}=\{x\in \R^n\setminus A_v:\rank\nabla v(x)<m\} .
\end{equation}
In previous joint papers of two of the authors with J.~Bourgain~\cite{BKK,BKK2} and
in~\cite{KK3,KK15} this definition of critical set was used and a
corresponding Dubovitski\u{\i}--Federer Theorem~B was established
for mappings of Sobolev class $\WW^k_\p(\R^n,\R^d)$. If, in
addition, the highest derivative $\nabla^kv$ belongs to the
Lorentz space $\LL_{\p,1}$ (in particular, if $k=n$ since
$\LL_{1,1}=\LL_1$), also the Luzin $N$--property with respect to
the $\p$--dimensional Hausdorff content was proven. It implies, in
particular, that the image of the set $A_v$ of
nondifferentiability points has zero measure, and consequently,
$\CC^1$-smoothness of almost all level sets follows.

In this paper we prove the Dubovitski\u{\i} Theorem~A for mappings
of the same Sobolev--Lorentz class $\WW^k_{\p,1}$ and with values
in $\R^d$ for arbitrary $d \geq m$.

\begin{ttt}\label{DST}{\sl
Let $k,m\in\{1,\dots,n\}$, \,$d\ge m$ \,and \,$v\in
\WW^{k}_{\p,1}(\R^n,\R^d)$. Then the equality
\begin{equation}\label{dub2}
\H^s (Z_{v,m}\cap v^{-1}(y))=0\qquad\mbox{ for \ $\H^m$-a.a.
}y\in\R^d
\end{equation}
holds, where as above $s =n-m-k+1$ and $Z_{v,m}$ denotes the set
of $m$--critical points of~$v$:\ \ $Z_{v,m}=\{x\in\R^n\setminus
A_v\,:\,\rank \nabla v(x)<m\}$.}
\end{ttt}

To the best of our knowledge the result is new even when the
mapping $v \colon \R^n \to \R^d$ is of class $\CC^k$ since we
allow here $m<d$ (compare with Theorem A). However, the main
thrust of the result is the extension to the Sobolev--Lorentz
context that we believe is essentially sharp. In this context we
also wish to emphasize that the result is in harmony with our
definition of critical set (recall that $\H^\p(A_v)=0$\,) and the
following new analog of the Luzin $N$-property:

\begin{ttt}\label{LDST}{\sl
Let $k,m\in\{1,\dots,n\}$, \,$d\ge m$ \,and \,$v\in
\WW^{k}_{\p,1}(\R^n,\R^d )$. Then for any set $A$ with
$\H^\p(A)=0$ we have
\begin{equation}\label{dub3}
\H^s (A\cap v^{-1}(y))=0\qquad\mbox{ for \ $\H^m$-a.a. }y\in\R^d ,
\end{equation}
where again $s=n-m-k+1$. }
\end{ttt}

We end this section with remarks about the possibility to
localize our results.

\begin{rem} \label{Sobolev_r}
We have formulated the results in the context of mappings $v
\colon \R^n \to \R^d$ for mere convenience. However, the reader
can easily check that the essence of our results is at the local
level and so they also apply to mappings $v \colon N \to D$ that
are locally of class $\WW^{k}_{\p,1}$ between a second countable
$n$--dimensional smooth manifold $N$ and a $d$--dimensional smooth
manifold $D$.
\end{rem}

\begin{rem} \label{Sobolev_rr}
Since for an open set~$U\subset \R^n$ of finite measure the
estimate~$\|1_U\cdot f\|_{\LL_{\p,1}}\le C_U \|f\|_{\LL_p(U)}$
holds  for~$p>\p$ (see, e.g., \cite[Theorem~3.8]{Maly2}), and,
consequently,
$$W^k_p(U)\subset W^k_{\p,1}(U)\subset W^k_\p(U),$$
the results of the above theorems~\ref{DST}--\ref{LDST} are in
particular valid for mappings~$v$ that are locally of class
$\WW^k_{p}$ with $p>\p=\frac{n}k$.
\end{rem}

\section{A Bridge between the theorems of
Dubovitski\u{\i} and Federer} Originally, the purpose of the
present paper was very concrete: to extend the Dubovitski\u{\i}
Theorem~A to the Sobolev context (since the
Federer--Dubovitski\u{\i} Theorem~B had been extended before
in~\cite{KK3,KK15}, see Introduction and Subsection~\ref{MSR}).
But when our paper was finished and ready for submission, the very
natural question arose. Theorem~A asserts that $\H^m$-almost all
preimages are small (with respect to $\H^s$-measure), and
Theorem~B claims that $\H^\b$-almost all preimages are empty.
Could we connect these results? More precisely, could we say
something about $\H^q$--almost all preimages for other values
of~$q$, say, for $q\in[m-1,\b]$? The affirmative answer is
contained in the next theorem.

\begin{ttt}\label{DST-q}{\sl
$k,m\in\{1,\dots,n\}$, \,$d\ge m$ \,and \,$v\in
\WW^{k}_{\p,1}(\R^n,\R^d)$. Then for any $q\in (m-1,\infty)$ the
equality
\begin{equation}\label{dub2-q}
\H^{\mu_q}(Z_{v,m}\cap v^{-1}(y))=0\qquad\mbox{ for \ $\H^q$-a.a.
}y\in\R^d
\end{equation}
holds, where
\begin{equation}\label{mmu}\mu_q:=s+k(m-q),\qquad\ s=n-m-k+1,
\end{equation} and
$Z_{v,m}$ again denotes the set of $m$-critical points of~$v$:\ \
$Z_{v,m}=\{x\in\R^n\setminus A_v\,:\,\rank\nabla v(x)\le m-1\}$.}
\end{ttt}

Let us note, that the behavior of the function $\mu_q$ is very
natural:
\begin{eqnarray}
\mu_q &=&0 \mbox{ for } q=\b=m-1+\frac{n-m+1}k \quad \mbox{(Dubovitski\u{\i}--Federer Theorem~B)} \nonumber\\
\mu_q&<&0 \mbox{ for } q>\b \quad  \mbox{[ibid.]} \nonumber\\
\mu_q&=& s\mbox{ for } q=m \quad \mbox{(Dubovitski\u{\i} Theorem A)} \nonumber\\
\mu_q &=& n-m+1 \mbox{ for } q=m-1.  \label{dub3-q}
\end{eqnarray}
The last value cannot be improved in view of the trivial example
of a linear mapping $L\colon \R^n\to\R^d$ of rank $m-1$.

Thus, Theorem~\ref{DST-q} contains all the previous theorems
(Morse--Sard, A,\,B,\,\ref{DST} and \ref{MS}\,) as particular
cases and it is new even for the~smooth case.

We emphasize the fact that in stating Theorem~\ref{DST-q} we
skipped the borderline case $q=m-1$, $\mu_q=n-m+1$. Of course, for
this case we cannot assert that $\H^{m-1}$--almost all preimages
in the $m$--critical set~$Z_{v,m}$ have zero $\H^{n-m+1}$--measure
as the above mentioned counterexample with a linear mapping $L \colon
\R^n\to\R^d$ of rank $m-1$ shows. But for this borderline case we
obtain instead the following analog of the classical coarea
formula:

\begin{ttt}\label{D-Coarea}{\sl
Let $n,d\in\N$, \,$m\in\{0,\dots,\min(n,d)\}$, \ and \ $v\in
\WW^{1}_{n,1}(\R^n,\R^d)$. Then for any Lebesgue measurable subset
$E$ of $Z_{v,m+1}=\{x\in\R^n\setminus A_v:\rank \nabla v(x)\le
m\}$ we have
\begin{equation}\label{dub-coar1}
\int\limits_{E} \! J_{m}v(x)\,\dd
x=\int\limits_{\R^d}\H^{n-m}(E\cap v^{-1}(y))\, \dd \H^{m}(y) ,
\end{equation}
where $J_{m}v(x)$ denotes the $m$--Jacobian of $v$ defined as the
product of the $m$ largest singular values of the matrix $\nabla
v(x)$. }
\end{ttt}

The proof relies crucially on the results of~\cite{Oht} and
\cite{Karm} that give criteria for the validity of the coarea
formula for Lipschitz mappings between metric spaces, see
also~\cite{BKA} and \cite{Maly3,MSZ}.

Thus, to study the level sets for the borderline case~$q=m-1$ in
Theorem~\ref{DST-q}, one must take $m'=m-1$ instead of~$m$ in
Theorem~\ref{D-Coarea}.

\begin{rem}\label{Coar-rem}
Note that for the case $m=n$ the formula (\ref{dub-coar1})
corresponds to the {\it area formula} whose validity for Sobolev
mappings supporting the $N$--property is well-known (see, e.g.,
\cite{Maly-area} and \cite{Maly1}, where the $N$-property was
established for mappings of class~$W^1_{n,1}$\,). But for $m< n$
the result is new even for~smooth mappings, since usually the
formula (\ref{dub-coar1}) is proved under the assumption~$d=m$
(see, e.g., \cite{MSZ} for Sobolev functions $W^1_p(\R^n,\R^m)$\,)
or, when $m<d$, under the assumption that the image $v(E)$ is a
$\H^m$-$\sigma$--finite set (e.g., \cite{Oht}, \cite{Karm}, see
also Theorem~\ref{D-Coarea-smooth-p} of the present paper\,).
\end{rem}

From the Coarea formula~(\ref{dub-coar1}) it follows directly,
that the set of $y\in\R^d$ where the integrand in the right-hand
side of~(\ref{dub-coar1}) is positive, is $\H^m$-$\sigma$--finite.
Indeed, from Theorem~\ref{D-Coarea} and \cite[Theorem~1.3]{Karm}
we obtain immediately the following more precise statement:

\begin{cor}\label{rect-cor}{\sl
Let $m\in\{0,\dots, \min(d,n) \}$  and $v\in
\WW^{1}_{n,1}(\R^n,\R^d)$. Then the set
$$
\biggl\{y\in\R^d\ :\ \H^{n-m}\bigl(Z_{v,m+1}\cap
v^{-1}(y)\bigr)>0\biggr\}
$$
is $\H^m$--rectifiable, i.e., it is a union of a set of
$\H^m$-measure zero and a~countable family of images $g_i(S_i)$ of
Lipschitz mappings $g_i \colon S_{i} \subseteq \R^m\to\R^d$. Here
again $Z_{v,m+1}=\{x\in\R^n\setminus A_v : \rank \nabla v(x)\le
m\}$. }
\end{cor}

\begin{rem}\label{Coar-rem2}In view of the embedding $\WW^{k}_{\p,1}(\R^n)\hookrightarrow
\WW^1_{n,1}(\R^n)$ for $k\in\{1,\dots,n\}$, $\p=\frac{n}k$ (see,
e.g., \cite[\S8]{Maly2}\,), the assertions of
Theorem~\ref{D-Coarea} and Corollary~\ref{rect-cor} are in
particular valid for the mappings $v\in \WW^k_{\p,1}(\R^n,\R^d)$,
i.e., under the conditions of Theorem~\ref{DST-q}.
\end{rem}

Again Theorems~\ref{DST-q} and \ref{D-Coarea} are in harmony with
our definition of critical set (recall that $\H^\p(A_v)=0$\,)
because of the following analog of the Luzin $N$-property:

\begin{ttt}\label{LDST-q}{\sl
Let $k\in\{1,\dots,n\}$, $\p = n/k$ and $v\in
\WW^{k}_{\p,1}(\R^n,\R^d)$. Then for every $p\in[\p,n]$, \
$q\in[0,p]$ and for any set $E\subset\R^n$ with $\H^p(E)=0$ we have
\begin{equation}\label{dub4-q}
\H^{p-q}(E\cap v^{-1}(y))=0\qquad\mbox{ for \ $\H^q$-a.a.
}y\in\R^d .
\end{equation}
In particular, \
\begin{equation}\label{dub4-q''}
\H^p\bigl(v(E)\bigr)=0\qquad\mbox{ whenever}\quad\H^p(E)=0,\quad
p\in[\p,n].
\end{equation} }
\end{ttt}

By a simple calculation we have for $q\in[0,\b]$ that
\begin{eqnarray}
\mu_q &=& n-m-k+1+k(m-q) \nonumber\\
&=& (\p-q)k+(m-1)(k-1)\ge\max(\p-q,0).\label{dub5-q}
\end{eqnarray}
Theorem~\ref{LDST-q} then yields

\begin{cor}\label{LDST-cor}{\sl
Let $k,m\in\{1,\dots,n\}$ \,and \,$v\in
\WW^{k}_{\p,1}(\R^n,\R^d)$. Then for every $q\in[0,+\infty)$ and
for any set $E$ with $\H^\p(E)=0$ we have
\begin{equation}\label{dub6-q}
\H^{\mu_q}(E\cap v^{-1}(y))=0\qquad\mbox{ for \ $\H^q$--a.a.
}y\in\R^d.
\end{equation}
}
\end{cor}

Consequently, for every $q\in [0,+\infty)$
\begin{equation}\label{dub7-q}
\H^{\mu_q}(A_v\cap v^{-1}(y))=0\qquad\mbox{ for \ $\H^q$--a.a. }y\in\R^d,
\end{equation}
where we recall that $A_v$ is the set of nondifferentiability
points of~$v$ (cf. with~(\ref{dub2-q})\,).
\bigskip

Finally, applying the $N$-property~(Theorem~\ref{LDST-q}) for
$p=n$, $q=m\le n$, we obtain
\begin{cor}\label{LDST-cor-cr}{\sl
Let $n$, $d \in \N$, \ $m\in[0,n]$, and $v\in
\WW^{1}_{n,1}(\R^n,\R^d)$. Then for any set $E$ of zero
$n$-Lebesgue measure $\Le^n(E)=0$ the identity
\begin{equation}\label{dub4-q-crr} \H^{n-m}(E\cap
v^{-1}(y))=0\qquad\mbox{ for \ $\H^m$-a.a. }y\in\R^d
\end{equation}
holds. }
\end{cor}

Thus the sets of $n$-Lebesgue measure zero (in particular, the set
of nondifferentiability points $A_v$\,) are negligible in the
Coarea formula~(\ref{dub-coar1}).

Finally, let us comment briefly on the proofs that merge ideas from our previous
papers~\cite{BKK2}, \cite{KK3,KK15} and \cite{HZ}. In particular, the joint
papers \cite{BKK,BKK2} by two of the authors with J.~Bourgain contain many of
the key ideas that allow us to consider nondifferentiable Sobolev mappings.
For the implementation of these ideas one relies on estimates for the Hardy--Littlewood
maximal function in terms of Choquet type integrals with respect to Hausdorff capacity.
In order to take full advantage of the Lorentz context we exploit the recent estimates
from~\cite{KK15} (recalled in Theorem~\ref{lb7} below, see also ~\cite{Ad} for the case~$p=1$).
As in \cite{BKK2} (and subsequently in \cite{KK3}) we also crucially use Y.~Yomdin's (see ~\cite{Yom})
entropy estimates of near critical values for polynomials (recalled in Theorem~\ref{lb8} below).

In addition to the above mentioned papers there is a growing number of papers on the topic, including
\cite{Alb,AS,Bates,BHS,Bu,Fig,HeHo15,CGTAM,Nor,PZ,Putten,Putten1}.
\bigskip

\noindent
{\em Acknowledgment.} P.H. was supported by NSF grant
DMS-1500647. \ M.K. was partially supported by the Russian
Foundations for Basic Research (Grant No.~14-01-00768-a) and by
the Dynasty Foundation.

\section{Preliminaries}
\label{prel}

\noindent
By  an \textbf{$n$--dimensional interval} we mean a closed
cube in $\R^n$ with sides parallel to the coordinate axes. If $Q$
is an $n$--dimensional cubic interval then we write $\ell(Q)$ for
its sidelength.

For a subset $S$ of $\R^n$ we write $\Le^n(S)$ for its outer
Lebesgue measure. The $m$--dimensional Hausdorff measure is
denoted by $\H^m$ and the $m$--dimensional Hausdorff content by
$\H^{m}_{\infty}$. Recall that for any subset $S$ of $\R^n$ we
have by definition
$$
\H^m (S)=\lim\limits_{\alpha\searrow 0}\H^m_\alpha (S) =
\sup_{\alpha >0} \H^{m}_{\alpha}(S),
$$
where for each $0< \alpha \leq \infty$,
$$
\H^m_\alpha (S)=\inf\left\{ \sum_{i=1}^\infty(\diam S_i)^m\ :\
\diam S_i\le\alpha,\ \ S \subset\bigcup\limits_{i=1}^\infty S_i
\right\}.
$$
It is well known that $\H^n(S)  = \H^n_\infty(S)\sim\Le^n(S)$ for
sets~$S\subset\R^n$.

To simplify the notation, we write $\|f\|_{\LL_p}$ instead of
$\|f\|_{\LL_p(\R^n)}$, etc.

The Sobolev space $\WW^{k}_{p} (\R^n,\R^d)$ is as usual defined as
consisting of those $\R^d$-valued functions $f\in \LL_p(\R^n)$
whose distributional partial derivatives of orders $l\le k$ belong
to $\LL_p(\R^n)$ (for detailed definitions and differentiability
properties of such functions see, e.g., \cite{EG}, \cite{M},
\cite{Ziem}, \cite{Dor}). Denote by $\nabla^k f$ the vector-valued
function consisting of all $k$-th order partial derivatives of $f$
arranged in some fixed order. However, for the case of first order
derivatives $k=1$ we shall often think of $\nabla f(x)$ as the
Jacobi matrix of $f$ at $x$, thus the $d \times n$ matrix whose
$r$-th row is the vector of partial derivatives of the $r$-th
coordinate function.

We use the norm
$$
\|f\|_{\WW^{k}_{p}}=\|f\|_{\LL_p}+\|\nabla
f\|_{\LL_p}+\dots+\|\nabla^kf\|_{\LL_p},
$$
 and unless otherwise specified all norms on the
spaces $\R^s$ ($s \in \N$) will be the usual euclidean norms.

Working with locally integrable functions, we always assume that
the precise representatives are chosen. If $w\in
L_{1,\loc}(\Omega)$, then the precise representative $w^*$ is
defined for {\em all} $x \in \Omega$ by
\begin{equation}
\label{lrule}w^*(x)=\left\{\begin{array}{rcl} {\displaystyle
\lim\limits_{r\searrow 0} \dashint_{B(x,r)}{w}(z)\,\dd z}, &
\mbox{ if the limit exists
and is finite,}\\
 0 \qquad\qquad\quad & \; \mbox{ otherwise},
\end{array}\right.
\end{equation}
where the dashed integral as usual denotes the integral mean,
$$
\dashint_{B(x,r)}{ w}(z) \, \dd
z=\frac{1}{\Le^n(B(x,r))}\int_{B(x,r)}{ w}(z)\,\dd z,
$$
and $B(x,r)=\{y: |y-x|<r\}$ is the open ball of radius $r$
centered at $x$. Henceforth we omit special notation for the
precise representative writing simply $w^* = w$.

We will say that $x$~is an $\LL_p$--Lebesgue point of~$w$ (and
simply a Lebesgue point when $p=1$), if
$$
\dashint_{B(x,r)}|{ w}(z)-w(x)|^p \,\dd z\to0\quad\mbox{ as }\quad
r\searrow 0.
$$

If $k<n$, then it is well-known that functions from Sobolev spaces
$\WW^{k}_{p}(\R^n)$ are continuous for $p>\frac{n}k$ and could be
discontinuous for $p\le \p=\frac{n}k$ (see, e.g., \cite{M,Ziem}).
The Sobolev--Lorentz space $\WW^{k}_{\p,1}(\R^n)\subset
\WW^{k}_{\p}(\R^n)$ is a refinement of the corresponding Sobolev
space that for our purposes turns out to be convenient. Among
other things functions that are locally in $\WW^{k}_{\p,1}$ on
$\R^n$ are in particular continuous.

Here we shall mainly be concerned with the Lorentz space
$\LL_{p,1}$, and in this case one may rewrite the norm as (see for
instance \cite[Proposition 3.6]{Maly2})
\begin{equation}\label{lor1}
\|f\|_{p,1}=
\int\limits_0^{+\infty}\bigl[\Le^n(\{x\in\R^n:|f(x)|>t\})\bigr]^{
\frac1p} \, \dd t.
\end{equation}
We record the following subadditivity property of the Lorentz norm
for later use.

\begin{lem}[see, e.g., \cite{rom1} or \cite{Maly2}]\label{asr3}
{\sl Suppose that $1\le p<\infty$ and $E=\bigcup_{j \in \N} E_j$,
where $E_j$ are measurable and mutually disjoint subsets of
$\R^n$. Then for all $f\in \LL_{p,1}$ we have
$$
\sum_j\|f\cdot 1_{E_j}\|^p_{\LL_{p,1}}\le\|f\cdot
1_E\|^p_{\LL_{p,1}},
$$
where $1_E$ denotes the indicator function of the set~$E$.}
\end{lem}

Denote by $\WW^{k}_{p,1}(\R^n)$ the space of all functions $v\in
\WW^k_p(\R^n)$ such that in addition the Lorentz norm~$\|\nabla^k
v\|_{\LL_{p,1}}$ is finite.

For a mapping $u \in \LL_1(Q,\R^d)$, $Q\subset\R^n$, $m\in\N$,
define the polynomial $P_{Q,m}[u] $ of degree at most~$m$ by the
following rule:
\begin{equation}
\label{0}\int_Qy^\alpha \left( u(y)-P_{Q,m}[u](y) \right) \,\dd
y=0
\end{equation}
for any multi-index $\alpha=(\alpha_1,\dots,\alpha_n)$ of length
$|\alpha|=\alpha_1+\dots+\alpha_n\le {m}$.  Denote $P_{Q}[u]=
P_{Q,k-1}[u]$.

The following well--known bound will be used on several occasions.

\begin{lem}[see, e.g.,\cite{KK15}]\label{lb3}{\sl
Suppose $v\in \WW^{k}_{\p,1}(\R^n,\R^d)$ { with
$k\in\{1,\dots,n\}$}. Then $v$ is a continuous mapping and for any
$n$-dimensional cubic interval $Q\subset \R^n$ the estimate
\begin{equation}
\label{1} \sup\limits_{y\in Q}\bigl|v(y)-P_{Q}[v](y)\bigr|\le
C\|1_Q\cdot \nabla^{k} v\|_{\LL_{\p,1}}
\end{equation}
holds, where $C$ is a constant depending on $n,d$ only. Moreover,
the mapping $v_{Q}(y)=v(y)-P_{Q}[v](y)$, $y\in Q$, can be extended
from~$Q$ to the whole of $\R^n$ such that the extension (denoted
again) $v_{Q}\in\WW^{k}_{\p,1}(\R^n,\R^d)$ and
\begin{equation}
\label{1'} \|\nabla^{k} v_{Q}\|_{\LL_{\p,1}(\R^n)}\le C_0
\|1_Q\cdot \nabla^{k} v\|_{\LL_{\p,1}},
\end{equation}
 where $C_0$ also depends on $n,d$ only. }
\end{lem}

\begin{cor}[see, e.g., \cite{KK3}]\label{me1}{\sl Suppose $v\in \WW^{k}_{\p,1}(\R^n,\R^d)$
{ with $k\in\{1,\dots,n\}$}. Then $v$ is a continuous mapping and
for any $n$-dimensional cubic interval $Q\subset \R^n$ the
estimates
\begin{equation}
\label{fme1} \diam v(Q)\le C\biggl(\frac{\|\nabla
v\|_{\LL_\p(Q)}}{\ell(Q)^{k-1}}+ \|1_Q\cdot\nabla^{k}
v\|_{\LL_{\p,1}}\biggr)\le C\biggl(\frac{\|\nabla
v\|_{\LL_p(Q)}}{\ell(Q)^{\frac{n}p-1}}+ \|1_Q\cdot\nabla^{k}
v\|_{\LL_{\p,1}}\biggr)
\end{equation}
hold for every~$p\in[\p,n]$. }
\end{cor}

The above results can easily be adapted to give that $v \in
\CC_{0}(\R^n )$, the space of continuous functions on $\R^n$ that
vanish at infinity (see for instance \cite[Theorem 5.5]{Maly2}).

 Let $\Me^\beta$ be the space of all nonnegative
Borel measures~$\mu$ on~$\R^n$ such that
\begin{equation}
\label{mu1} | \! | \! |\mu | \! | \!
|_{\beta}=\sup_{I\subset\R^n}\ell(I)^{- \beta}\mu(I)<\infty,
\end{equation}
where the supremum is taken over all $n$--dimensional  cubic
intervals $I\subset\R^n$ and $\ell(I)$ denotes side--length
of~$I$. We need the following important strong-type estimates for
maximal functions (it was proved in \cite{KK15} based on classic
results of D.R.~Adams~\cite{Ad} and some new analog of the trace
theorem for Riesz potentials of Lorentz functions for the limiting
case $q=p$, see Theorems~0.2--0.4 and Corollary~2.1
in~\cite{KK15}).

\begin{ttt}[\cite{KK15}]\label{lb7}{\sl
Let $p\in(1, \infty)$, $k,l\in\{1,\dots,n\}$, $l\le k$,
$(k-l)p<n$. Then for any function $f\in \WW^{k}_{p,1}(\R^n)$ the
estimates
\begin{eqnarray}
\label{adcb3} \|\nabla^l f\|^p_{L_p(\mu)} \leq C| \! | \! |\mu |
\! | \! |_\beta\Vert \nabla^k
f\Vert^p_{\LL_{p,1}}\quad\forall\mu\in\Me^\beta,\\ \label{max1}
\int_0^\infty\H^{\beta}_\infty( \{x\in\R^n : \M \bigl(|\nabla^l
f|^p\bigr)(x)\ge t \})\,\dd t \le C\Vert \nabla^k
f\Vert^p_{\LL_{p,1}}
\end{eqnarray}
hold, where $\beta=n-(k-l)p$, the constant $C$ depends on $n,k,p$
only, and
$$
\M f(x)=\sup\limits_{r>0}  \dashint_{B(x,r)} \! |f(y)|\,\dd y
$$
is the usual Hardy--Littlewood maximal function of~$f$. }
\end{ttt}

The result is true also for $p=1$, $k>l$ and is in this case due to
D.R.~Adams~\cite{Ad}.

For a subset $A$ of ${\R}^m$ and $\varepsilon>0$ the
$\varepsilon$--entropy of $A$, denoted by $\E(\varepsilon,A)$, is
the minimal number of closed balls of radius $\varepsilon$
covering $A$. Further, for a linear map~$L\colon \R^n\to\R^d$ we denote by
$\lambda_j(L)$, $j=1,\dots,d$, its singular values
arranged in decreasing order: \ $\lambda_1(L) \ge
\lambda_2(L) \ge\dots\ge\lambda_d(L)$. Geometrically the
singular values are the lengths of the semiaxes of the, possibly degenerate,
ellipsoid $L( \partial B(0,1))$. We recall that the singular values of $L$
coincide with the eigenvalues repeated according to multiplicity
of the symmetric nonnegative linear map~$\sqrt{LL^{\ast}}\colon
\R^d \to \R^d$. Also for a mapping~$f\colon \R^n\to\R^d$ that is
approximately differentiable at $x \in \R^n$
put~$\lambda_j(f,x)=\lambda_j(d_xf)$, where by $d_xf$ we denote
the  approximate differential of $f$ at~$x$. The next result is
the second basic ingredient of our proof.

\begin{ttt}[\cite{Yom}]\label{lb8}{\sl
For any polynomial $P\colon \R^n\to\R^d$ of degree at most $k$,
for each ball $B\subset \R^n$ of radius $r>0$, and any number
$\varepsilon >0$ we have that
$$
\E\bigl(\varepsilon r,\{P(x):x\in B,\ \lambda_1
\le1+\varepsilon,\dots,\lambda_{m-1} \le1+\varepsilon, \lambda_m
\le\varepsilon,\dots,\lambda_d \le\e\}\bigr)
$$
$$\le C_Y \bigl(1+\e^{1-m}\bigr),
$$
where the~constant $C_Y$ depends on $n,d,k,m$ only and for
brevity we wrote $\lambda_{j}= \lambda_{j}(P,x)$.}
\end{ttt}

The application of Theorem~\ref{lb7} is facilitated through the
following simple estimate (see for instance Lemma~2 in \cite{Dor},
cf. with~\cite{BH}\,).

\begin{lem}\label{lb10} {\sl Let $u\in\WW^{1}_{1}(\R^n,\R^d)$. Then for
any ball $B\subset \R^n$ of radius $r>0$ and for any number
$\varepsilon >0$ the estimate
$$
\diam (\{u(x):x\in B,\ (\M \nabla u)(x)\le\varepsilon\})\le C_M
\varepsilon r
$$
holds, where $C_{M}$ is a constant depending on $n,d$ only.}
\end{lem}

We need also the following approximation result.

\begin{ttt}[see Theorem~2.1 in \cite{KK15}]\label{Th_ap}{\sl Let $p\in(1, \infty)$,
$k,l\in\{1,\dots,n\}$, $l\le k$, $(k-l)p<n$. Then for any
$f\in\WW^{k}_{p,1}(\R^n)$ and for each $\varepsilon>0$ there exist
an open set $U\subset\R^n$ and a function $g\in \CC^l(\R^n)$ such
that

\begin{itemize}
\item[(i)] \,$\H^{n-(k-l)p}_\infty (U)<\varepsilon$;

\item[(ii)] \,each point $x\in\R^n\setminus U$ is
{an~$\LL_{p}$-Lebesgue point for~$\nabla^j f$,\, $j=0,\dots,l$;}

\item[(iii)] \,$f\equiv g$, $\nabla^jf \equiv \nabla^jg$ on
$\R^n\setminus U$ for $j=1,\dots,l$.
\end{itemize}
}
\end{ttt}

Note that in the analogous theorem for the case of Sobolev
mappings $f\in\WW^{k}_{p}(\R^n)$ the assertion (i) should be
replaced by
\bigskip

(i') $\mathcal B_{k-l,p}(U)<\varepsilon$ if $l<k$,
\bigskip

\noindent
where $\mathcal{B}_{\alpha,p}(U)$ denotes the Bessel
capacity of the set~$U$ (see, e.g., Chapter~3 in \cite{Ziem} or
\cite{BHS}\,).

Recall that for $1<p<\infty$ and $0<n - \alpha p <n$ the smallness
of $\H^{n - \alpha p}_\infty(U)$ implies the smallness of
$\Be_{\alpha,p}(U)$, but that the opposite is false since
$\Be_{\alpha,p}(U)=0$ whenever $\H^{n - \alpha p}(U)<\infty$. On
the other hand, for $1<p<\infty$ and $0<n - \alpha p < \beta\le n$
the smallness of $\Be_{\alpha,p}(U)$ implies the smallness of
$\H^\beta_\infty(U)$ (see, e.g., \cite{Aikawa}). So the usual
assertion (i') is essentially weaker than~(i).

\section{Luzin $N$- and Morse--Sard properties for Sobolev--Lorentz
mappings}\label{MSR}

In this section we briefly recall some theorems
from~\cite{KK3,KK15} which we need. The following result is an
analog of the Luzin $N$-property with respect to the Hausdorff
content.

\begin{ttt}[\cite{KK3,KK15}]\label{th3.3}{\sl
Let $k\in\{1,\dots,n\}$, \,$q\in[\p,n]$, and $v\in
\WW^{k}_{\p,1}(\R^n,\R^d)$. Then for each $\varepsilon>0$ there
exists $\delta>0$ such that for any set $E\subset\R^n$ \ if \
$\H^{q}_\infty(E)<\delta$, then $\H_\infty^{q}(v(E))<\varepsilon$.
In particular, $\H^{q}(v(E))=0$ whenever $\H^{q}(E)=0$.}
\end{ttt}

The next asertion is the precise analog of the
Dubovitski\u{\i}--Federer theorem~B (see
Introduction~\ref{Introd}\,) which includes the Morse--Sard
result.

\begin{ttt}[\cite{KK3,KK15}]
\label{MS} {\sl  If $k,m\in\{1,\dots,n\}$, $\Omega$ is an open
subset of $\R^n$, and $v\in \WW^{k}_{\p,1,{\rm loc}}(\Omega
,\R^d)$, then $\H^\b(v(Z_{v,m}))=0$.}
\end{ttt}

Recall that in our notation
\begin{equation}\label{dd1}
\p=\frac{n}{k},\quad s =n-m-k+1, \quad \b=m+\frac{s}k=\p+(m-
1)\bigl(1-k^{-1}\bigr),
\end{equation}
and $Z_{v,m}=\{x\in\Omega\,:\,\rank \nabla v(x)<m\}$.

Finally, here we recall some differentiability properties of
Sobolev--Lorentz functions.

\begin{ttt}[\cite{KK3,KK15}]
\label{Th_dif} {\sl  Let $k\in\{1,\dots,n\}$ and
$v\in\WW^{k}_{\p,1}(\R^n,\R^d)$. Then there exists a Borel set
$A_v\subset\R^n$ such that $\H^\p(A_v)=0$ and for any $x\in
\R^n\setminus A_v$ the function $v$ is differentiable (in the
classical Fr\'{e}chet sense) at $x$, furthermore,  the classical
derivative coincides with $\nabla v(x)$ \ ($x$ is
an~$\LL_{\p}$-Lebesgue point for~$\nabla v$). }
\end{ttt}

Really the last assertion of the Theorem~--- that $\H^{\p}$-almost
all points $x\in\R^n$ are the $\LL_{\p}$-Lebesgue points for the
gradient~$\nabla v$~--- follows from Theorem~\ref{Th_ap}~(ii).

The case $k=1$, $\p =n$ of the Theorem~\ref{Th_dif} is a classical
result due to Stein \cite{St} (see also \cite{Maly1}), and for
$k=n$, $\p=1$ the result is due to Dorronsoro~\cite{Dor}.

Theorem~\ref{Th_dif} admits the following generalization.

\begin{ttt}[\cite{KK3,KK15}] \label{Th_dif2} {\sl  Let
$k,l\in\{1,\dots,n\}$, $l\leq k$,  and
$v\in\WW^{k}_{\p,1}(\R^n,\R^d)$. Then there exists a Borel set
$A_{v,l}\subset\R^n$ such that $\H^{l\p}(A_{v,l})=0$ and each
point $x\in \R^n\setminus A_{v,l}$ is an~$\LL_{p}$-Lebesgue point
for~$\nabla^j f$,\, $j=0,\dots,l$, moreover, the function $v$ is
$l$-times differentiable (in the classical Fr\'{e}chet--Peano
sense) at $x$, i.e.,
$$
\lim\limits_{r\searrow 0} \sup\limits_{y\in
B(x,r)\setminus\{x\}}\frac{\bigl|v(y)-T_{v,l,x}(y)\bigr|}{|x-y|^l}=0,
$$
where $T_{v,l,x}(y)$ is the Taylor polynomial of order $l$ for $v$
centered at~$x$.}
\end{ttt}

Note that the Taylor polynomial of order $l$ for $v$ centered
at~$x$ is well defined $\H^{l\p}$-a.e. by Theorem~\ref{Th_ap}.

\section{Proofs of the main results}

\subsection{Proof of the Luzin type~$N$-property}

In this subsection we are going to prove Theorem~\ref{LDST-q} and
as a consequence Theorem~\ref{LDST}. Now fix $n\in\N$,
$k\in\{1,\dots,n\}$, \ $p\in[\p,n]$  \ and \ $q\in[0,p]$. Denote
in this subsection
\begin{equation}\label{muu1}
\mu=p-q.
\end{equation}
Fix also a mapping $v\in\WW^k_{\p,1}(\R^n,\R^d)$. For a set
$E\subset\R^n$ define the set function
\begin{equation}\label{dd5}
\Phi(E)=\inf\limits_{E\subset\bigcup_\alpha
D_\alpha}\sum\limits_\al\bigl(\diam D_\alpha\bigr)^\mu\bigl[\diam
v(D_\alpha)\bigr]^q,
\end{equation}
where the infimum is taken over all countable families of compact
sets $\{D_\alpha\}_{\alpha\in \N}$ such that
$E\subset\bigcup_\alpha D_\alpha$. By Theorem~\ref{alaFalc} (see
Appendix), \ $\Phi(\cdot)$ is a countably subadditive set-function
with the property
\begin{equation}\label{dd6} \Phi(E)=0\ \boldsymbol{\Rightarrow}\
\biggl[\H^\mu\bigl(E\cap v^{-1}(y)\bigr)=0\quad\mbox{for
$\H^q$-almost all }y\in\R^d\biggr].
\end{equation}
Thus the assertion of Theorem~\ref{LDST-q} amounts to
\begin{equation}\label{dd7}
\Phi(E)=0\qquad\mbox{ whenever }\ \H^p(E)=0.
\end{equation}
The proof of this follows the ideas of~\cite{KK3}.

\noindent By a \textbf{dyadic interval} we understand  a cubic
interval of the form
$[\frac{k_1}{2^l},\frac{k_1+1}{2^l}]\times\dots\times[\frac{k_n}{2^l},\frac{k_n+1}{2^l}]$,
where $k_i,l$ are integers. The following assertion is
straightforward, and hence we omit its proof here.

\begin{lem}\label{lemD}{\sl
For any $n$--dimensional cubic interval $J\subset\R^n$ there exist
dyadic intervals $Q_1,\dots,Q_{2^n}$ such that $J\subset
Q_1\cup\dots\cup Q_{2^n}$ and
$\ell(Q_1)=\dots=\ell(Q_{2^n})\le2\ell(J)$. }
\end{lem}

Let  $\{ Q_\alpha \}_{\alpha \in A}$ be a family of
$n$-dimensional dyadic intervals. We say that the family $\{
Q_\alpha \}$ is {\bf regular}, if for any $n$-dimensional dyadic
interval $Q$ the estimate
\begin{equation}\label{q8}
\ell(Q)^{{p}}\ge\sum\limits_{\alpha : Q_\alpha\subset
Q}\ell(Q_\alpha)^{{{p}}}
\end{equation}
holds. Since dyadic intervals are either nonoverlapping or
contained in one another, (\ref{q8}) implies that any regular
family $\{ Q_\alpha \}$ must in particular consist of
nonoverlapping intervals.

\begin{lem}[see Lemma~2.3 in \cite{BKK2}]\label{lb5.1}{\sl
Let $\{ Q_\alpha \}$ be a family of $n$--dimensional dyadic
intervals. Then there exists
 a regular family $\{ J_\beta \}$ of $n$--dimensional
dyadic intervals such that $\bigcup_\alpha Q_\alpha\subset
\bigcup_\beta J_\beta$ and
$$
\sum\limits_{\beta}\ell(J_\beta)^{{{p}}}\le\sum\limits_{\alpha}\ell(
Q_\alpha)^{{{p}}}.
$$
}\end{lem}

\begin{lem}[see {\color{blue} Lemma~2.11} in \cite{KK15} and Lemma~2.4 in \cite{KK3}]\label{Thh3.3}{\sl
Let $v \in \WW^{k}_{\p ,1}(\R^n , \R^d )$. For each
$\varepsilon>0$ there exists $\delta=\delta(\varepsilon,v)>0$ such
that for any regular family $\{ Q_\alpha \}$ of $n$--dimensional
dyadic intervals we have if
\begin{equation}
\label{oxe0} \sum_\alpha\ell(Q_\alpha)^{{p}}<\delta,
\end{equation} then
\begin{equation}
\label{oxe1} \sum_\alpha\biggl[\bigl\|1_{Q_\alpha}\cdot\nabla^{k}
v\bigr\|^{{{p}}}_{\LL_{\p,1}}+\frac1{\ell(Q_\alpha)^{n-{p}}}\int_{Q_\alpha}|\nabla
v|^{{{p}}}\biggr]<\varepsilon.
\end{equation}
}\end{lem}

\begin{proof}[Proof of Theorem~\ref{LDST-q}]
Let $\H^{p}(E)=0$. Take $\varepsilon>0$ and
$\delta=\delta(\varepsilon ,v)<1$ from Lemma~\ref{Thh3.3}. Take
also the regular family $\{ Q_\alpha\}$ of $n$--dimensional dyadic
intervals such that $E\subset\bigcup\limits_\alpha Q_\alpha$ and
\begin{equation}
\label{ass2} \sum_\alpha\ell(Q_\alpha)^{{{p}}}<\delta
\end{equation}
where the existence of such family follows directly from
Lemmas~\ref{lemD}  and \ref{lb5.1}. Then by Lemma~\ref{Thh3.3} the
estimate~(\ref{oxe1}) holds. Denote $r_\alpha=\ell(Q_\al)$. By
estimate~(\ref{fme1}),
\begin{equation} \label{pln1}
\bigl[\diam v(Q_\al)\bigr]^q\le C\biggl(\frac{\|\nabla
v\|^q_{\LL_{p}(Q_\al)}}{r_\al^{(\frac{n}p-1)q}}+
\|1_{Q_\al}\cdot\nabla^{k} v\|^q_{\LL_{\p,1}}\biggr).
\end{equation}
Therefore, by definition of~\,$\Phi(E)$ \ (see (\ref{dd5})\,), we
have
\begin{eqnarray}\label{ass7}
\Phi(E) &\leq& C \sum\limits_\alpha r_\al^\mu\biggl(\frac{\|\nabla
v\|^q_{\LL_{p}(Q_\al)}}{r_\al^{(\frac{n}p-1)q}}+
\|1_{Q_\al}\cdot\nabla^{k} v\|^q_{\LL_{\p,1}}\biggr)\nonumber \\
&\overset{\mbox{\footnotesize\color{red}H\"{o}lder ineq.}}\leq&
c\biggl(\sum\limits_\al
r_\al^{\frac{\mu{p}}{{p}-q}}\biggr)^{\frac{{p}-q}{p}}\cdot\biggl[
\sum\limits_\alpha\biggl(\frac1{\ell(Q_\alpha)^{n-{p}}}\int_{Q_\alpha}|\nabla
v|^{{{p}}}+\|1_{Q_\al}\cdot\nabla^{k}v\|^{p}_{\LL_{\p,1}}\biggr)\biggr]^{\frac{q}{p}}\nonumber\\
&\overset{\mbox{\footnotesize(\ref{muu1}), (\ref{oxe1})}}\leq&
c\biggl(\sum\limits_\al r_\al^{p}\biggr)^{\frac{{p}-q}{p}}\cdot \e^{\frac{q}{p}}\nonumber\\
&\overset{\mbox{\footnotesize(\ref{ass2})}}\leq& c
\delta^{\frac{{p}-q}{p}}\cdot\e^{\frac{q}{p}}.
\end{eqnarray}
Since $\e>0$ and $\delta>0$ are arbitrary small, (\ref{ass7})
\,turns to the equality $\Phi(E)=0$ and by~(\ref{dd6}) the
required assertion is proved.
\end{proof}

\begin{rem}\label{int_opt}
Note that the regularity assumptions in the last theorem are
sharp: for example, the Luzin $N$--property fails in general for
continuous mappings $v\in \WW^1_n(\R^n,\R^n)$ (here $k=1$,
$\p={p}=n=q$, $\mu=0$), see, e.g., \cite{Maly}. The sharpness of
our assumptions for general order Sobolev spaces, though not on
the Sobolev--Lorentz scale, is also a consequence of the recent
and interesting results in \cite{HeHo15}. See also \cite{Kauf} for
earlier results in this direction.
\end{rem}

\subsection{{Dubovitski\u{\i}} theorem for Sobolev mappings}\label{MSP}

Fix integers $k,m\in\{1,\dots,n\}$, $d \geq m$ and a mapping $v\in
\WW^{k}_{\p,1}(\R^n,\R^d)$. Then, by Theorem~\ref{Th_dif}, there
exists a Borel set $A_v$ such that $\H^\p(A_v)=0$ and all points
of the complement $\R^n \setminus A_v$ are {$\LL_\p$-}Lebesgue
points for the weak gradient~$\nabla v$. Moreover, we can arrange
that $v$ is differentiable (in the classical Fr\'{e}chet sense) at
every point~$x \in \R^n \setminus A_v$ with derivative $\nabla
v(x)$ (so the classical derivative coincides with the precise
representative of the weak gradient at $x$).

Denote $Z_{v,m}=\{x\in \Omega \setminus A_v:\rank\nabla v(x)<m\}$.
Fix a~number $$q\in[m-1,\b).$$ Denote in this subsection
\begin{equation}\label{muu3}
\mu=\mu_q=n-m-k+1+(m-q)k.
\end{equation} Since $q<\b=m-1+\frac{n-m+1}k$, we
have $\mu>0$.

The purpose here is to prove the assertion of the \textit{bridge
Dubovitski\u{\i}--Federer Theorem}~\ref{DST-q} which is
equivalent (by virtue of Theorem~\ref{alaFalc}) to
\begin{equation}\label{ev1}
\Phi(Z_{v,m})=0 \qquad\mbox{ if }\ \ q>m-1,
\end{equation}
where for each fixed $q\in[m-1,\b)$ we denoted
\begin{equation}\label{dd5-dd}
\Phi(E)=\inf\limits_{E\subset\bigcup_\alpha
D_\alpha}\sum\limits_\al\bigl(\diam D_\alpha\bigr)^\mu\bigl[\diam
v(D_\alpha)\bigr]^q.
\end{equation}
As indicated the infimum is taken over all countable families of
compact sets $\{D_\alpha\}_{\alpha\in \N}$ such that
$E\subset\bigcup_\alpha D_\alpha$. Note that the case $q=\b$,
$\mu_q=0$ was considered in \cite{KK3,KK15} (see also
Subsection~\ref{MSR}\,), so we shall omit it here.

Before embarking on the detailed proof we make some preliminary
observations that allow us to make a few simplifying assumptions.
In view of our definition of critical set we have that
$$
Z_{v,m} = \bigcup_{j \in \N} \{ x \in Z_{v,m}: \, | \nabla v(x)|
\leq j \} .
$$
Consequently we only need to prove that~$\Phi(Z'_v)=0$ for
$q\in(m-1,\b)$, where
$$
Z'_v=\{x\in Z_{v,m}:|\nabla v(x)|\le1\} .
$$

\noindent For convenience, below we use the notation
$\|f\|_{\LL_{\p,1}(I)}$ instead of $\|1_I\cdot f\|_{\LL_{\p,1}}$.
The following lemma contains the main step in the proof of
Theorems~\ref{DST-q} and \ref{D-Coarea}.

\begin{lem}\label{lb11}{\sl
Let $q\in[m-1,\b)$. Then for any $n$-dimensional dyadic interval
$I\subset\R^n$ the estimate
\begin{equation}
\label{6} \Phi(Z'_v\cap I)\leq C\bigl(\ell(I)^\mu\,\|\nabla^{k}
v\|^{q}_{\LL_{\p,1}(I)}+\ell(I)^{\mu+m-1}\,\|\nabla^{k}
v\|^{q-m+1}_{\LL_{\p,1}(I)}\bigr)
\end{equation}
holds, where the constant $C$ depends on $n,m,k,d$ only.}
\end{lem}

\begin{proof}
By virtue of (\ref{1'}) it suffices to prove that
\begin{equation}
\label{6corr} \Phi(Z'_v\cap I)\leq
C\bigl(\ell(I)^\mu\,\|\nabla^{k}
v_I\|^{q}_{\LL_{\p,1}(\R^n)}+\ell(I)^{\mu+m-1}\,\|\nabla^{k}
v_I\|^{q-m+1}_{\LL_{\p,1}(\R^n)}\bigr)
\end{equation}
for the mapping $v_{I}$ defined in Lemma~\ref{lb3}, where
$C=C(n,m,k,d)$ is a constant.

Fix an $n$-dimensional dyadic interval $I\subset\R^n$ and recall
that $v_{I}(x)=v(x)-P_{I}(x)$ for all $x\in I$. Denote
$$
\sigma=\|\nabla^{k} v_{I}\|_{\LL_{\p,1}},\qquad r=\ell(I),
$$
and for each $j\in \mathbb Z$
$$
E_j=\bigl\{x\in I: (\M |\nabla v_{I}|^{\p})(x)\in(2^{j-
1},2^{j}]\bigr\} \quad \mbox{ and } \quad
\delta_j=\H^{\p}_\infty(E_j).
$$
Then by Theorem~\ref{lb7} (applied for the case $p=\p=\frac{n}k$,
$l=1$, $\beta=\p$\,),
\begin{equation}
\label{tr-d1} \sum\limits_{j=-\infty}^\infty \delta_j2^j\leq
C\sigma^\p
\end{equation}
for a constant $C$ depending on $n,m,k,d$ only. By the definition
of the~Hausdorff measure, for each $j\in\mathbb Z$ there exists a
family of balls $B_{ij}\subset\R^n$ of radii $r_{ij}$ such that
\begin{equation}
\label{res1} E_j\subset \bigcup\limits_{i=1}^\infty B_{ij} \mbox{\
\ \ and\ \ \ }\sum\limits_{i=1}^\infty r^{\p}_{ij}\le c\delta_j.
\end{equation}

Of course, using standard covering lemmas we can assume without
loss of generality that the concentric balls $\tilde B_{ij}$ with
radii $\frac15r_{ij}$ are disjoint, hereby follows in particular
that
\begin{equation}
\label{ev2} \sum\limits_{i\in\N,\ j\in\Z} r^{n}_{ij}\le C
r^n\qquad\mbox{ and }\qquad\sum\limits_{i\in\N,\ j\in\Z}
r^{\lambda n}_{ij}\leq Cr^{\lambda n} \qquad\forall\lambda\geq 1.
\end{equation}

Denote
$$
Z_j=Z'_v\cap E_j \quad \mbox{ and } \quad Z_{ij}=Z_j\cap B_{ij}.
$$
By construction $Z'_v\cap I=\bigcup_{j}Z_{j}$ \ and \
$Z_j=\bigcup_{i}Z_{ij}$. Put
$$
\e_*=\frac1{r}\|\nabla^kv_I\|_{\LL_{\p,1}}=\frac{\sigma}{r},
$$
and let $j_*$ be the integer satisfying
$\e_*^{\p}\in(2^{j_*-1},2^{j_*}]$. Denote
$Z_*=\bigcup_{j<j_*}Z_{j}$, \ $Z_{**}=\bigcup_{j\ge j_*}Z_{j}$.
Then by construction
$$
Z'_v\cap I=Z_*\cup Z_{**},\quad Z_*\subset \{x\in Z'_v\cap I:(\M
|\nabla v_{I}|^{\p})(x)<\e_*^{\p}\}.
$$

Since $\nabla P_{I}(x)=\nabla v(x)-\nabla v_{I}(x)$, $|\nabla
v_I(x)|\le 2^{j/\p}$, $|\nabla v(x)|\leq 1$, and
$\lambda_m(v,x)=0$ for $x\in Z_{ij}$ , we have\footnote{Here we
use the following elementary fact: for any linear maps $L_1 \colon
\R^n\to\R^d$ and $L_2 \colon \R^n\to\R^d$ the estimates
$\lambda_l(L_2+L_2)\le \lambda_l(L_1)+\|L_2\|$ hold for all
$l=1,\dots,m$, see, e.g., \cite[Proposition 2.5 (ii)]{Yom}.}
$$
Z_{ij}\subset \bigl\{x\in B_{ij}: \lambda_1(P_{I},x)\le
1+2^{j/\p},\dots,\lambda_{m-1}(P_{I},x)\le 1+2^{j/\p},\
\lambda_m(P_{I},x)\le 2^{j/\p}\bigr\}.
$$
Applying Theorem~\ref{lb8} and Lemma~\ref{lb10} to mappings
$P_{I}$, \ $v_{I}$, respectively, with $B=B_{ij}$ and
$\e=\varepsilon_j =2^{j/\p}$, we find a finite family of balls
$T_s\subset\R^d$, $s=1,\dots,s_j$ with $s_j \leq
C_{Y}(1+\e_j^{1-m})$, each of radius $(1+C_{M})\e_jr_{ij}$,  such
that
$$
\bigcup\limits_{s=1}^{s_j}T_s \supset v(Z_{ij}).
$$
Therefore, for every $j\ge j_*$ we have
\begin{equation}\label{Iom1}
\Phi(Z_{ij})\overset{\footnotesize(\ref{dd5-dd})}\le C_1
s_j\e_j^{q}r^{q+\mu}_{ij}=C_2(1+\e_j^{1-m})2^{\frac{jq}{\p}}r_{ij}^{q+\mu}\le
C_2(1+\e_*^{1-m})2^{\frac{jq}{\p}}r_{ij}^{q+\mu},
\end{equation}
where all the constants $C_\alpha$ above depend on $n,m,k,d$ only.
By the same reasons, but this time applying Theorem~\ref{lb8} and
Lemma~\ref{lb10} with $\e=\e_*$ and instead of the balls $B_{ij}$
we take a ball $B\supset I$ with radius $\sqrt{n}r$, we have
\begin{eqnarray}
\Phi(Z_{*})\le C_3(1+\e_*^{1-m})\e_*^q r^{q+\mu}
&\overset{\mbox{\footnotesize\color{red}definitions}}=&
C_3(1+\sigma^{1-m}r^{m-1})\sigma^q r^{\mu}\nonumber\\
&=& C_3\bigl(r^{\mu} \sigma^q
+r^{\mu+m-1}\sigma^{q-m+1}\bigr).\label{Iom2}
\end{eqnarray}
From (\ref{Iom1}) we get immediately
\begin{equation}\label{Iom2'}
\Phi(Z_{**})\le C_2(1+\e_*^{1-m})\sum\limits_{j\ge j_*}
\sum\limits_i2^{\frac{jq}{\p}}r_{ij}^{q+\mu}.
\end{equation}
Further estimates splits into the two possibilities.

{\bf Case I.} \ $q\ge \p$. Then
\begin{eqnarray}
\Phi(Z_{**}) &\overset{(\ref{Iom2'})}\leq &
C_2(1+\e_*^{1-m})\biggl(\sum\limits_{j\ge j_*}\sum\limits_i2^{j}r_{ij}^{(q+\mu)\frac{\p}{q}}\biggr)^{\frac{q}\p}\nonumber\\
&\leq& C_2(1+\e_*^{1-m})r^\mu\biggl(\sum\limits_{j\geq j_*}\sum\limits_i2^{j}r_{ij}^{\p}\biggr)^{\frac{q}\p}\nonumber\\
&\overset{(\ref{res1})}\leq& C_4(1+\e_*^{1-m})r^\mu\biggl(\sum\limits_{j\geq j_*}2^{j}\delta_j\biggr)^{\frac{q}\p}\nonumber\\
&\overset{(\ref{tr-d1})}\leq& C_5(1+\e_*^{1-m})r^\mu\sigma^q=
C_5\bigl(r^{\mu} \sigma^q +r^{\mu+m-1}\sigma^{q-m+1}\bigr)
\label{Iomm2}
\end{eqnarray}

{\bf Case II.}  \ $q< \p$. Recalling (\ref{muu3}) we get by an
elementary calculation
\begin{eqnarray}\label{fimm1}
\frac{\mu\p}{\p-q} &=& n\cdot\frac{n-qk+[mk-m-k+1]}{n-qk}\nonumber\\
&=& n\cdot\frac{n-qk+(m-1)(k-1)}{n-qk} \geq n,
\end{eqnarray}
therefore,
\begin{eqnarray}
\Phi(Z_{**}) &\overset{(\ref{Iom2'}),\
\mbox{\footnotesize\color{red}H\"{o}lder ineq.}}\leq&
C_2(1+\e_*^{1-m})\biggl(\sum\limits_{j\geq
j_*}\sum\limits_i2^{j}r_{ij}^{\p}\biggr)^{\frac{q}\p}\cdot
\biggl(\sum\limits_{j\geq j_*}\sum\limits_ir_{ij}^{\frac{\mu\p}{\p-q}}\biggr)^{\frac{\p-q}{\p}}\nonumber\\
&\overset{(\ref{res1}),\ (\ref{tr-d1})}\leq&
C_6(1+\e_*^{1-m})\sigma^q \biggl(\sum\limits_{j\geq j_*}
\sum\limits_ir_{ij}^{\frac{\mu\p}{\p-q}}\biggr)^{\frac{\p-q}{\p}}\nonumber\\
&\overset{(\ref{ev2}),\ (\ref{fimm1})}\leq& C_6(1+\e_*^{1-m})\sigma^q r^\mu\nonumber\\
&=& C_6\bigl(r^{\mu} \sigma^q
+r^{\mu+m-1}\sigma^{q-m+1}\bigr).\label{Iomm3}
\end{eqnarray}
Now for both cases (I) and (II) we have by (\ref{Iomm2}),
(\ref{Iomm3}) that $\Phi(Z_{**})\le C\bigl(r^{\mu} \sigma^q
+r^{\mu+m-1}\sigma^{q-m+1}\bigr)$, and, by virtue of the earlier
estimate~(\ref{Iom2}),  we conclude that
$$
\Phi(Z'_v\cap I)=\Phi(Z_*\cup Z_{**})\leq
\Phi(Z_*)+\Phi(Z_{**})\le C\bigl(r^{\mu} \sigma^q
+r^{\mu+m-1}\sigma^{q-m+1}\bigr).
$$
The lemma is proved.
\end{proof}

\begin{cor}
\label{cor00} {\sl  Let $q\in[m-1,\b)$. Then for any
$\varepsilon>0$ there exists $\delta>0$ such that for any subset
$E$ of $\R^n$ we have $\Phi(Z'_v\cap E)\le\varepsilon$ provided
$\Le^n(E)\le\delta$. In particular, $\Phi(Z_{v,m}\cap E)=0$
whenever $\Le^n(E)=0$.}
\end{cor}

\begin{proof}
We start by recording the following elementary identity
(see~(\ref{muu3})\,):
\begin{equation}\label{fimm3}
\frac{(\mu+m-1)\p}{\p-q+m-1}=n.
\end{equation} Let $\Le^n(E)\le\delta$, then we can find a family of
nonoverlapping $n$-dimensional dyadic intervals~$I_\alpha$ such
that $E\subset\bigcup_\alpha I_\alpha$ and
$\sum\limits_\alpha\ell^n(I_\alpha)<C\delta$. Of course, for
sufficiently small $\delta$ the estimates
\begin{equation}\label{eles1}\|\nabla^{k}
v\|_{\LL_{\p,1}(I_\alpha)}<1,\qquad \ell(I_\alpha)\le
\delta^\frac1n
\end{equation} are fulfilled for every~$\alpha$. Denote
\begin{equation}\label{ds5}r_\alpha=\ell(I_\alpha),\qquad\sigma_\alpha=\|\nabla^{k}
v\|_{\LL_{\p,1}(I_\alpha)},\qquad\sigma=\|\nabla^{k}
v\|_{\LL_{\p,1}}.\end{equation} In view of Lemma~\ref{lb11} we
have
\begin{equation}\label{oxMS1}\Phi(E)\le C\sum_\alpha
r_\alpha^{\mu+m-1}\sigma_\alpha^{q-m+1}+C\sum_\alpha
r_\alpha^{\mu} \sigma_\alpha^q.
\end{equation}
Now let us estimate the first sum. Since by our assumptions
$$
q<\b=m-1+\frac{n-m+1}k\le m-1+\p\ \mbox{ hence }\ \p> q-m+1
$$
we have
\begin{eqnarray}\label{oxMS3}
\sum_\alpha r_\alpha^{\mu+m-1}\sigma^{q-m+1}_\alpha
&\overset{\mbox{\footnotesize\color{red}H\"{o}lder ineq.}}\leq &
C\biggl(\sum\limits_\alpha\sigma_\alpha^\p\biggr)^{\frac{q-m+1}\p}\cdot \biggl(\sum\limits_\alpha r_{\alpha}^{\frac{(\mu+m-1)\p}{\p-q+m-1}}\biggr)^{\frac{\p-q+m-1}{\p}}\nonumber\\
&\overset{(\ref{fimm3}),\
\mbox{\footnotesize\color{red}Lemma~\ref{asr3}}}\leq& C'
\sigma^{q-m+1}\cdot \biggl(\Le^n(E)\biggr)^{\frac{\p-q+m-1}{\p}}.
\end{eqnarray}
The estimates of the second sum are again handled by consideration
of two separate cases.

{\bf Case I.} \ $q\ge \p$. Then
\begin{equation}\label{le-Iomm2}
\sum_\alpha
r_\alpha^{\mu}\sigma_\alpha^q\overset{(\ref{eles1})}\le
\delta^{\frac{\mu}n}\sum_\alpha\sigma^\p_\alpha\overset{\mbox{\footnotesize\color{red}Lemma~\ref{asr3}}}\le
\sigma^\p\cdot\delta^{\frac{\mu}n}.
\end{equation}

{\bf Case II.}  \ $q< \p$. Then
\begin{eqnarray}\label{le-Iomm3}
\sum_\alpha r_\alpha^{\mu}\sigma_\alpha^q
&\overset{\mbox{\footnotesize\color{red}H\"{o}lder ineq.}}\leq &
\biggl(\sum\limits_{\alpha}\sigma_\alpha^\p\biggr)^{\frac{q}\p}\cdot \biggl(\sum\limits_{\alpha}r_{\alpha}^{\frac{\mu\p}{\p-q}}\biggr)^{\frac{\p-q}{\p}}\nonumber\\
&\overset{\mbox{\footnotesize\color{red}Lemma~\ref{asr3}},\
(\ref{fimm1})}\leq & \sigma^q\delta^{\mu}.
\end{eqnarray}
Now for both cases (I) and (II) we have by
(\ref{oxMS1})--(\ref{le-Iomm3}) that $\Phi(E)\le h(\delta)$, where
the function $h(\delta)$ satisfies the condition
$h(\delta)\searrow 0$ as $\delta\searrow 0$. The lemma is proved.
\end{proof}
By Theorem~\ref{Th_ap}~(iii) (applied to the case $k=l$\,), our
mapping~$v$ coincides with a mapping $g\in \CC^{k}(\R^n,\R^d)$ off
an exceptional set of small $n$--dimensional Lebesgue measure.
This fact, together with Corollary~\ref{cor00} and
Dubovitski\u{\i} Theorem~A, finishes the proof of
Theorem~\ref{DST} for the case $d=m$. But since
Theorem~\ref{DST-q} was not proved for $\CC^k$--smooth
mappings\footnote{Even Theorem~A was not proved for $\R^d$--valued
mappings.}, we have to do this step now.

\begin{lem}
\label{Dub-smooth} {\sl  Let $q\in(m-1,\b)$ \ and \ $g\in
\CC^{k}(\R^n,\R^d)$. Then
\begin{equation}
\label{ff-ds} \Phi_g(Z_{g,m})=0,
\end{equation} where $\Phi_g$ is calculated by the same
formula~(\ref{dd5-dd}) with $g$ instead of~$v$ and
$Z_{g,m}=\{x\in\R^n:\rank\nabla g(x)<m\}$. }
\end{lem}

\begin{proof}
We can assume without loss of generality that $g$ has compact
support and that $| \nabla g(x)| \leq 1$ for all $x \in \R^n$. We
then clearly have that $g\in \WW^{k}_{\p,1}(\R^n,\R^d)$, hence we
can in particular apply the above results to $g$. The following
assertion plays the key role:

\noindent ($*$) \ {\sl For any $n$-dimensional dyadic interval
$I\subset\R^n$ the estimate
$$
\Phi(Z_{g,m}\cap I)\le C\bigl(\ell(I)^\mu\,\|\nabla^{k}\bg_I
\|^{q}_{\LL_{\p,1}(I)}+\ell(I)^{\mu+m-1}\,\|\nabla^{k}
\bg_I\|^{q-m+1}_{\LL_{\p,1}(I)}\bigr)
$$
holds, where the constant $C$ depends on $n,m,k,d$ only, and we
denoted
$$
\nabla^{k}\bg_I(x)=\nabla^{k}g(x)-\dashint\limits_I\nabla^kg(y)\,
\dd y.
$$}
The proof of ($*$) is almost the same as that of
Lemma~\ref{lb11}, with evident modifications (we need to take the
approximation polynomial~$P_I(x)$ of degree $k$ instead of~$k-1$,
etc.).

By elementary facts of the Lebesgue integration theory, for an
arbitrary family of nonoverlapping $n$-dimensional dyadic
intervals~$I_\alpha$ one has
\begin{equation}
\label{6-ds}\sum\limits_\alpha \|\nabla^{k}
\bg_{I_\alpha}\|^\p_{\LL_{\p,1}(I_\alpha)}\to0\qquad\mbox{ as \
}\sup\limits_\alpha\ell(I_\alpha)\to 0
\end{equation}
The proof of this estimate is really elementary since now
$\nabla^kg$ is continuous and compactly supported function, and,
consequently, is uniformly continuous and bounded.

From ($\ast$) and (\ref{6-ds}), repeating the arguments of
Corollary~\ref{cor00}, using the assumptions on $g$ and taking
$$
\sigma_\alpha=\|\nabla^{k}
\bg_{I_\alpha}\|_{\LL_{\p,1}(I_\alpha)},\qquad
\sigma=\sum\limits_\alpha \sigma_\alpha^\p
$$
in definitions~(\ref{ds5}), we obtain that $\Phi_g(Z_{g,m})<\e$
for any $\e>0$, hence the sought conclusion~(\ref{ff-ds}) follows.
\end{proof}

By Theorem~\ref{Th_ap}~(iii) (applied to the case $k=l$\,), the
investigated mapping~$v$ equals a mapping $g\in
\CC^{k}(\R^n,\R^d)$ off an exceptional set of small
$n$--dimensional Lebesgue measure. This fact together with
Lemma~\ref{Dub-smooth} readily implies

\begin{cor}[cp.~with \cite{DeP}]
\label{Th_ap2} {\sl  Let $q\in(m-1,\b)$. Then there exists a set
$\widetilde Z_{v}$ of $n$-dimensional Lebesgue measure zero such
that $\Phi(Z'_v\setminus \widetilde Z_{v})=0$. In particular,
$\Phi(Z'_v)=\Phi(\widetilde Z_{v})$.}
\end{cor}
From Corollaries~\ref{cor00} and \ref{Th_ap2} we conclude that
$\Phi(Z'_v)=0$, and this concludes the proof of
Theorem~\ref{DST-q}.

\subsection{The proof of the Coarea formula}

Fix $v\in \WW^{1}_{n,1}(\R^n,\R^d)$\,). Applying Lemma~\ref{lb11}
for $k=1$, $\p=n$, $\mu=n-m+1$ and $q=m-1$, and afterwards making
the shift of indices $(m-1)\to m$, we obtain the following key
estimate:

{\sl Let $m\in\{0,\dots,n-1\}$. Then for any $n$-dimensional
dyadic interval $I\subset\R^n$ the estimate
\begin{equation}
\label{6-coar'} \Phi(Z'_v\cap I)\le
C\bigl(\ell(I)^{n-m}\,\|\nabla^{k}
v\|^{m}_{\LL_{\p,1}(I)}+\ell(I)^{n}\bigr)
\end{equation}
holds, where $Z'_{v}=\bigl\{x\in \Omega \setminus A_v:\rank\nabla
v(x)\le m,\quad|\nabla v(x)|\le1\bigr\}$, the constant $C$ depends
on $n,m,d$ only, and
\begin{equation}\label{dopolnit}
\Phi(E)=\inf\limits_{E\subset\bigcup_\alpha
D_\alpha}\sum\limits_\al\bigl(\diam
D_\alpha\bigr)^{n-m}\bigl[\diam v(D_\alpha)\bigr]^m.
\end{equation}}

This implies (by the same arguments as in the proof of
Corollary~\ref{cor00}) that for any measurable set $E\subset\R^n$
with $\Le^n(E)<\infty$ the inequality
\begin{equation}
\label{6-coar}\Psi(Z'_v\cap E)<\infty\end{equation} holds, where
$\Psi(E)$ is defined as
\begin{equation}\label{cs-finit-mdd5}
\Psi(E)=\lim\limits_{\delta\to0}\inf\limits_{\footnotesize{\begin{array}{lcr}E\subset\bigcup_\alpha
D_\alpha,\\
\diam D_\alpha\le\delta\end{array}}}\sum\limits_\al\bigl(\diam
D_\alpha\bigr)^{n-m}\bigl[\diam v(D_\alpha)\bigr]^m,
\end{equation}
 here the infimum is taken over all countable families of
compact sets $\{D_\alpha\}_{\alpha\in \N}$ such that
$E\subset\bigcup_\alpha D_\alpha$ and $\diam D_\alpha\le\delta$
for all~$\alpha$.

By Theorem~\ref{finit-alaFalc}, the bound (\ref{6-coar}) implies
the validity of the following assertion:
\begin{equation}\label{cpr1}
\mbox{the set \ }\biggl\{y\in\R^d:\H^{n-m}\bigl(E\cap Z'_{v}\cap
f^{-1}(y)\bigr)>0\biggr\}\quad\mbox{ is $\H^{m}$ $\sigma$-finite}.
\end{equation}
Since
$$
Z_{v,m+1}=\bigl\{x\in \Omega \setminus A_v:\rank\nabla v(x)\le
m\bigr\}=\bigcup\limits_j \biggl\{x\in Z_{v,m+1}: |\nabla v(x)|\le
j\biggr\},
$$
we infer from~(\ref{cpr1}) that in fact
\begin{equation}\label{cpr2}
\mbox{the set \ }\biggl\{y\in\R^d:\H^{n-m}\bigl( Z_{v,m+1}\cap
f^{-1}(y)\bigr)>0\biggr\}\quad\mbox{ is $\H^{m}$ $\sigma$-finite}.
\end{equation}
Next we  prove that the sets where \ $\rank \nabla v\le m-1$ are
negligible in the coarea formula.

\begin{lem}
\label{negl-z-sets} {\sl  The equality
\begin{equation}\label{dd6-coar} \H^{n-m}\bigl(Z_{v,m}\cap
v^{-1}(y)\bigr)=0\quad\mbox{for $\H^{m}$-almost all }y\in\R^d
\end{equation}
holds, where $Z_{v,m}=\{x\in\R^n\setminus A_v:\rank\nabla v(x)\le
m-1\}$ is the set of $m$-critical points.}
\end{lem}

\begin{proof}
We apply Theorem~\ref{DST-q} with the parameters $q=m$, $k=1$,
$\p=n$. Then by (\ref{dub2-q})
\begin{equation}\label{dd6-coar4} \H^{\mu_q}\bigl(Z_{v,m}\cap
v^{-1}(y)\bigr)=0\quad\mbox{for $\H^m$-almost all }y\in\R^d,
\end{equation}
where $\mu_q=n-m-k+1+{\color{blue}(m-q)}k= n-m$. The last identity
taken together with~(\ref{dd6-coar4}) concludes the proof.
\end{proof}
In the papers \cite{Oht,Karm} the authors identified criteria for
the validity of the Coarea formula for Lipschitz mappings. The
following result is particularly useful to us.

\begin{ttt}[see, e.g., Theorem 1.4 in \cite{Karm}]\label{D-Coarea-smooth-p}{\sl
Let $m\in\{0,1,\dots,n\}$,  and $g\in \CC^1 (\R^n,\R^d)$. Suppose
that the set $E\subset \R^n$ is measurable and $\rank\nabla
g(x)\equiv m$ for all $x\in E$. Assume also that the set $g(E)$ is
$\H^{m}$-$\sigma$-finite. Then the coarea formula
\begin{equation}\label{dub-coar1-smmoth}
\int\limits_{E}J_{m}g(x)\, \dd x=\int\limits_{\R^d}\H^{n-m}(E\cap
g^{-1}(y))\, \dd \H^{m}(y)
\end{equation}
holds, where $J_{m}g(x)$ denotes the $m$--Jacobian of $g$. }
\end{ttt}
Of course, (\ref{cpr2}) and (\ref{dd6-coar}) are in particular
valid also for $\CC^k$--smooth mappings. So from
Theorem~\ref{D-Coarea-smooth-p} and
properties~(\ref{cpr2})--(\ref{dd6-coar}) we obtain the following
result which surprisingly is new even in this smooth case.

\begin{ttt}\label{D-Coarea-smooth}{\sl
Let $m\in\{0,\dots,n\}$  and $g\in \CC^1 (\R^n,\R^d)$. Then for
any measurable set $E\subset Z_{g,m+1}=\{x\in\R^n:\rank\nabla
g(x)\le m\}$ the coarea formula
\begin{equation}\label{dub-coar1-smmoth-22}
\int\limits_{E}J_{m}g(x)\,\dd x=\int\limits_{\R^d}\H^{n-m}(E\cap
g^{-1}(y))\, \dd \H^{m}(y)
\end{equation}
holds, where $J_{g,m}(x)$ again denotes the $m$--Jacobian of $g$.
}
\end{ttt}
By Theorem~\ref{Th_ap}~(iii) (applied to the case $k=l=1$\,), the
investigated mapping~$v\in \WW^{1}_{n,1}(\R^n,\R^d)$ coincides
with a smooth mapping $g\in \CC^{1}(\R^n,\R^d)$ off a set of small
$n$--dimensional Lebesgue measure. This fact together with
Theorems~\ref{D-Coarea-smooth}, \ref{alaFalc} and
Corollary~\ref{cor00} easily imply the required assertion of
Theorem~\ref{D-Coarea}.

\section{Appendix}

Fix numbers $n,d\in\N$, $\mu\in(0,n]$, $q\in(0,d]$, and a
continuous function $f \colon \R^n\to\R^d$. For a set
$E\subset\R^n$ define the set function
\begin{equation}\label{mdd5}
\Phi(E)=\inf\limits_{E\subset\bigcup_\alpha
D_\alpha}\sum\limits_\al\bigl(\diam D_\alpha\bigr)^\mu\bigl[\diam
v(D_\alpha)\bigr]^q,
\end{equation}
where the infimum is taken over all countable families of compact
sets $\{D_\alpha\}_{\alpha\in \N}$ such that
$E\subset\bigcup_\alpha D_\alpha$.

This section is devoted to the proof of following assertion:
\begin{ttt}\label{alaFalc} {\sl The above defined set function $\Phi(\cdot)$ is countably subadditive
and
\begin{equation}\label{mdd6} \Phi(E)=0\ \Rightarrow\
\biggl[\H^\mu\bigl(E\cap f^{-1}(y)\bigr)=0\quad\mbox{for
$\H^q$-almost all }y\in\R^d\biggr].
\end{equation}}
\end{ttt}
We start by recalling the following technical fact from
\cite{Dav56}:

\begin{lem}\label{lmes1} {\sl For any set $E\subset\R^n$, if $E=\bigcup\limits_{i=1}^\infty
E_i$ and $E_i\subset E_{i+1}$ for all $i\in\N$, then
\begin{equation}\label{mes1}
\H^\mu_\infty(E)=\lim\limits_{i\to\infty} \H^\mu_\infty(E_i).
\end{equation}}
\end{lem}

\begin{proof}[Proof of Theorem~\ref{alaFalc}]
The first assertion is evident. Let us prove the second one,
i.e.,~the implication~(\ref{mdd6}). Without loss of generality we
can assume that $f$ is compactly supported, and more specifically
that $f^{-1}(y)$ is a compact subset of the closed unit ball
$\overline{B(0,1)}$ for every $y\in\R^d\setminus\{0\}$.

Let $E\subset\R^n$ and assume that $\Phi(E)=0$. Without loss of
generality we can assume that $0\notin f(E)$ and
$$
E=\bigcap\limits_{j=1}^\infty\bigcup\limits_{i=1}^\infty D_{ij},
$$
where $D_{ij}$ are compact sets in~$\R^n$ and
\begin{equation}\label{mes2}
\sum\limits_{i=1}^\infty\bigl(\diam D_{ij}\bigr)^\mu\bigl[\diam
f(D_{ij})\bigr]^q\,\underset{\footnotesize j\to\infty}\to0.
\end{equation}
Of course, then $E$ is a Borel set. Suppose that the
assertion~(\ref{mdd6}) is false, then we can assume without loss
of generality that there exists a set $\cF\subset f(E)$ such that
\begin{equation}\label{mes3}
\H^q(\cF)>0\qquad\mbox{ and }\qquad\H_\infty^\mu\bigl(E\cap
f^{-1}(y)\bigr)\ge\frac52\quad\mbox{ for all }y\in \cF.
\end{equation}
Unfortunately, we can not assume right now that the set $\cF$ is
Borel, so we need some careful preparations.

Denote $E_{kj}=\bigcup\limits_{i=1}^k D_{ij}$, \
$E_j=\bigcup\limits_{i=1}^\infty D_{ij}$. In this notation
$E=\bigcap\limits_{j=1}^\infty E_j$. Evidently, all these sets are
Borel. By Lemma~\ref{lmes1},
\begin{equation}\label{mes5}
\H^\mu_\infty(E_j\cap
f^{-1}(y))=\lim\limits_{k\to\infty}\H^\mu_\infty(E_{kj}\cap
f^{-1}(y))\qquad \mbox{for each } y\in f(E_j).
\end{equation}
Denote further $F_{kj}=f(E_{kj})$. Fix an~arbitrary point $y$ with
the~property
$$
\H^\mu(E_{kj}\cap f^{-1}(y)\,)\leq 1.
$$
Since $E_{kj}$ is a compact set, the set $E_{kj}\cap f^{-1}(y)$ is
compact as well. Then it follows by elementary means that the sets
$E_{kj}\cap f^{-1}(z)$ lie in the $\e$-neighborhood of the set
$E_{kj}\cap f^{-1}(y)$, where $\e\searrow 0$ as $z\to y$, \,$z\in
f(E_{kj})$. Therefore, there exists $\delta=\delta(y)>0$ such that
\begin{equation}\label{mes6}
\H^\mu_\infty(E_{kj}\cap f^{-1}(z))\le2\qquad\mbox{ if
}|z-y|<\delta.
\end{equation}
Hence, there exists a relatively open set $\tF_{kj}\subset F_{kj}$
(i.e., $\tF_{kj}$ is open in the induced topology of the
set~$F_{kj}$\,) such that
\begin{equation}\label{mes7}
\{y\in \R^d: \H^\mu_\infty(E_{kj}\cap f^{-1}(y))\leq
1\}\,\subset\,\tF_{kj}\,\subset\,\{y\in \R^d:
\H^\mu_\infty(E_{kj}\cap f^{-1}(y))\leq 2\}.
\end{equation}
Since by construction $F_{kj}$ is a~compact set and $\tF_{kj}$ is
relatively open in~$F_{kj}$, we conclude that {\bf the set
$\tF_{kj}$ is Borel} (this fact plays an~important role here).
Further, since $E_{kj}\subset E_j$, we have for each $k \in \N$,
$$
\{y\in \R^d: \H^\mu_\infty(E_{j}\cap f^{-1}(y))\leq 1\}\,\subset
\,\{y\in \R^d: \H^\mu_\infty(E_{kj}\cap f^{-1}(y))\leq
1\}\,\subset\,\tF_{kj}
$$
and therefore,
\begin{equation}\label{mes8}
\{y\in \R^d: \H^\mu_\infty(E_{j}\cap f^{-1}(y))\leq
1\}\,\subset\,\tF_{j},
\end{equation}
where we denote $\tF_j=\bigcap\limits_{k=1}^\infty \tF_{kj}$. On
other hand, (\ref{mes5}) and the second inclusion in~(\ref{mes7})
imply $\tF_{j}\subset\{y\in \R^d: \H^\mu_\infty(E_{j}\cap
f^{-1}(y))\leq 2\}$, so we have
\begin{equation}\label{mes9}
\{y\in \R^d: \H^\mu_\infty(E_{j}\cap f^{-1}(y))\leq
1\}\,\subset\,\tF_{j}\,\subset\,\{y\in \R^d:
\H^\mu_\infty(E_{j}\cap f^{-1}(y))\leq 2 \}.
\end{equation}
Denote now $\ttF_j=f(E_j)\setminus\tF_j$. Then we can rewrite
(\ref{mes9}) as
\begin{equation}\label{mes10}
\{y\in \R^d: \H^\mu_\infty(E_{j}\cap
f^{-1}(y))>2\}\,\subset\,\ttF_{j}\,\subset\,\{y\in \R^d:
\H^\mu_\infty(E_{j}\cap f^{-1}(y))>1\}.
\end{equation}
Since $E\subset E_j$, we have from (\ref{mes3}) that $\cF\subset
\{y\in \R^d: \H^\mu_\infty(E_{j}\cap f^{-1}(y))>2\}\subset\ttF_j$
for all $j\in\N$, therefore
\begin{equation}\label{mes11}
\cF\subset\,\ttF,
\end{equation}
where we denote $\ttF=\bigcap\limits_{j=1}^\infty \ttF_{j}$. On
the other hand, the second inclusion in~(\ref{mes10}) yields
\begin{equation}\label{mes12}
\ttF\subset\{y\in \R^d: \H^\mu_\infty(E_{j}\cap f^{-1}(y))>1\}
\end{equation}
for each $j \in \N$. Since $\ttF$ {\bf is a Borel set} and by
(\ref{mes11}), (\ref{mes3}) the inequalities $\H^q(\ttF)\ge
\H^q(\cF)>0$ hold, by \cite[Corollary~4.12]{Falcon} there exists a
Borel set $G\subset\ttF$ and a~positive constant $b\in\R$ such
that $0<\H^q(G)<\infty$ and
\begin{equation}\label{mes13}
\H^q(G\cap B(y,r))\le b\,r^q
\end{equation}
for any ball $B(y,r)=\{z\in\R^d: |z-y|< r\}$ with the center~$y\in
G$. Of course, by (\ref{mes12})
\begin{equation}\label{mes15}
G\subset\{y\in \R^d: \H^\mu_\infty(E_{j}\cap f^{-1}(y))>1\}
\end{equation}
for all $j \in \N$. For $S\subset\R^n$ consider the set function
\begin{equation}\label{mmdd5}
\widetilde\Phi(S)=\int\limits_{G}^*\H^\mu_\infty\bigl(S\cap
f^{-1}(y)\bigr)\, \dd \H^q(y),
\end{equation}
where $\int\limits^*$ means the {\it upper} integral\footnote{We
use the notion of the upper integral since we do not know whether
or not the function $y\mapsto\H^\mu_\infty\bigl(S\cap
f^{-1}(y)\bigr)$ is measurable.}. By standard facts of Lebesgue
integration theory,  $\widetilde\Phi(\cdot)$ is a countably
subadditive set--function (see, e.g., \cite{EG}, \cite{HZ}\,).

From (\ref{mes2}) and (\ref{mes13}) it follows that
\begin{eqnarray*}
\sum\limits_{i=1}^\infty\bigl(\diam D_{ij}\bigr)^\mu\bigl[\diam
f(D_{ij})\bigr]^q &\geq& c\sum\limits_{i=1}^\infty\bigl(\diam D_{ij}\bigr)^\mu\H^q\bigl[G\cap f(D_{ij})\bigr]\\
&\geq& C\sum\limits_{i=1}^\infty\widetilde\Phi(D_{ij})\geq
C\,\widetilde\Phi(E_j).
\end{eqnarray*}
Consequently, $\widetilde\Phi (E_{j}) \to 0$ as $j \to \infty$. On
the other hand, from (\ref{mes15}) and (\ref{mmdd5}) we conclude
$$
\widetilde\Phi(E_j)\geq \int\limits_{G}^*d\H^q(y)=\H^q(G)>0,
$$
which is the desired contradiction. The proof of the
Theorem~\ref{alaFalc} is finished.
\end{proof}

\subsection{$\H^{q}$-$\sigma$-finiteness of the image}

Now again fix numbers $n,d\in\N$, $\mu\in(0,n]$, $q\in(0,d]$ and a
continuous mapping $f \colon \R^n\to\R^d$. We define the set
function by letting for a set $E \subset \R^n$,
\begin{equation}\label{finit-mdd5}
\Psi(E)=\lim\limits_{\delta\to0}\inf\limits_{\footnotesize{\begin{array}{lcr}E\subset\bigcup_\alpha
D_\alpha,\\
\diam D_\alpha\le\delta\end{array}}}\sum\limits_\al\bigl(\diam
D_\alpha\bigr)^\mu\bigl[\diam f(D_\alpha)\bigr]^q,
\end{equation}
where the infimum is taken over all countable families of compact
sets $\{D_\alpha\}_{\alpha\in \N}$ such that
$E\subset\bigcup_\alpha D_\alpha$ and $\diam D_\alpha\le\delta$
for all~$\alpha$.

This subsection is devoted to the following assertion:
\begin{ttt}\label{finit-alaFalc} {\sl The above defined $\Psi(\cdot)$ is a countably subadditive
set--function and for any $\lambda>0$ the estimate
\begin{equation}\label{finit-mdd6}
\H^q\bigl( \{ y\in\R^d: \H^\mu\bigl(E\cap f^{-1}(y)\bigr)\geq
\lambda \} \bigr)\leq 5\frac{\Psi(E)}{\lambda}
\end{equation}
holds.}
\end{ttt}

\begin{proof}
The first assertion is evident and we focus on proving the
estimate~(\ref{finit-mdd6}). Without loss of generality we can
assume that $f^{-1}(y)$ is a compact subset of the closed unit
ball $\overline{B(0,1)}$ for every $y\in\R^d\setminus\{0\}$. Let
$E\subset\R^n$ and
$$
\Psi(E)=\sigma<\infty.
$$
Without loss of generality assume also that $0\notin f(E)$ and
$$
E=\bigcap\limits_{j=1}^\infty\bigcup\limits_{i=1}^\infty D_{ij},
$$
where $D_{ij}$ are compact sets in~$\R^n$ satisfying
\begin{equation}\label{finit-mes2}
\sum\limits_{i=1}^\infty\bigl(\diam D_{ij}\bigr)^\mu\bigl[\diam
f(D_{ij})\bigr]^q\,\underset{\footnotesize j\to\infty}\to\sigma,
\end{equation}
and
\begin{equation}\label{finit-mes-im}
\diam D_{ij}+\diam f(D_{ij})\leq \frac1j.
\end{equation}
Of course, $E$ is a Borel set. Fix $\lambda>0$ and take a set
$\cF\subset f(E)$ such that
\begin{equation}\label{finit-mes3}
\H_\infty^\mu\bigl(E\cap
f^{-1}(y)\bigr)\ge\frac52\lambda\quad\mbox{ for all }y\in \cF.
\end{equation}
Further we assume that
\begin{equation}\label{finit-mes3'}
\H^q(\cF)>0
\end{equation}
since if $\H^q(\cF)=0$, there is nothing to prove. Denote
$E_j=\bigcup\limits_{i=1}^\infty D_{ij}$. Repeating almost
verbatim the arguments from the proof of the previous
Theorem~\ref{alaFalc}, we can construct a Borel
set~$\ttF\subset\R^d$ such that
\begin{equation}\label{finit-mes11}
\cF\subset\,\ttF \subset\{y\in \R^d: \H^\mu_\infty(E_{j}\cap
f^{-1}(y))>\lambda\}
\end{equation}
for each $j \in \N$. Since $\ttF$ {\bf is a Borel set} and since,
by (\ref{finit-mes11}) and (\ref{finit-mes3'}), the inequalities
$\H^q(\ttF)\geq \H^q(\cF)>0$ hold, we deduce by
\cite[Theorem~4.10]{Falcon} the existence of a Borel set
$G\subset\ttF$ such that $0<\H^q(G)<\infty$. Put
\begin{equation}\label{finit-mes13}
G_l=\biggl\{x\in G\ :\ \H^{q}(G\cap B(x,r))\le 2r^q\qquad\forall
r\in(0,1/l)\biggr\}.
\end{equation}
Then by construction all the sets $G_l$ are Borel, $G_l\subset
G_{l+1}$, moreover, by \cite[Theorem~2~of~\S2.3]{EG} we have
$$
\H^q\biggl[G\setminus\biggl(\bigcup\limits_{l=1}^\infty
G_l\biggr)\biggr]=0
$$
and consequently,
 \begin{equation}\label{finit-mes17}
\H^q(G)=\lim\limits_{l\to\infty}\H^q(G_l).
\end{equation}
For $S\subset\R^n$ consider the set function
\begin{equation}\label{finit-mmdd5}
\Psi_l(S)=\int\limits_{G_l}^*\H^\mu_\infty\bigl(S\cap
f^{-1}(y)\bigr)\, \dd \H^q(y),
\end{equation}
where $\int\limits^*$ means the {\it upper} integral\footnote{We
use the notion of upper integral as it is unclear whether the
function $y\mapsto\H^\mu_\infty\bigl(S\cap f^{-1}(y)\bigr)$ is
measurable.}. By routine arguments of Lebesgue integration theory
it follows that  $\Psi(\cdot)$ is a countably subadditive
set-function (see, e.g., \cite{EG}, \cite{HZ}\,).

From (\ref{finit-mes2}), (\ref{finit-mes-im}) and
(\ref{finit-mes13}) it follows for $j> l$ that
\begin{eqnarray}\label{finit-mes17'}
\sum\limits_{i=1}^\infty\bigl(\diam D_{ij}\bigr)^\mu\bigl[\diam
f(D_{ij})\bigr]^q &\geq&
\frac12\sum\limits_{i=1}^\infty\bigl(\diam
D_{ij}\bigr)^\mu\,\H^q\bigl[G_l\cap
f(D_{ij})\bigr]\nonumber\\
&\geq& \frac12\sum\limits_{i=1}^\infty\Psi_l(D_{ij})\geq
\frac12\,\Psi_l(E_j).
\end{eqnarray}
On the other hand, the second inclusion in~(\ref{finit-mes11})
implies
\begin{equation}\label{finit-mes18}
\Psi_l(E_j)\geq
\lambda\int\limits_{G_l}^*d\H^q(y)=\lambda\H^q(G_l).
\end{equation}
From (\ref{finit-mes17'}), (\ref{finit-mes18}), (\ref{finit-mes2})
we infer
\begin{equation}\label{finit-mes19}
\H^q(G_l)\le \frac{2\sigma}{\lambda},
\end{equation}
and therefore, by (\ref{finit-mes17}),
\begin{equation}\label{finit-mes20}
\H^q(G)\le \frac{2\sigma}{\lambda}.
\end{equation}
Since this estimate is true for {\bf any} Borel set $G\subset
\ttF$ with $\H^q(G)<\infty$, and since $\ttF$ is Borel as well, we
infer from \cite[Theorem~4.10]{Falcon} that
\begin{equation}\label{finit-mes21}
\H^q(\ttF)\le \frac{2\sigma}{\lambda}.
\end{equation}
In particular, by the inclusion $\cF\subset\ttF$, this implies
\begin{equation}
\H^q(\cF)\le \frac{2\sigma}{\lambda},
\end{equation}
or in other words,
\begin{equation}
\H^q\bigl(y\in\R^d: \H^\mu\bigl(E\cap f^{-1}(y)\bigr)\geq
\frac52\lambda\bigr)\le2\frac{\Psi(E)}{\lambda}.
\end{equation}
The proof of Theorem~\ref{finit-alaFalc} is complete.
\end{proof}

\noindent Sobolev Institute of Mathematics, Acad.~Koptyuga pr., 4,
and Novosibirsk State University, Pirogova Str. 2, 630090
Novosibirsk, Russia\\
e-mail: \textit{korob@math.nsc.ru}

\bigskip

\noindent
Mathematical Institute, Andrew Wiles Building,
University of Oxford, Oxford OX2 6GG, England\\
e-mail: \textit{kristens@maths.ox.ac.uk}

\bigskip

\noindent
 Department of Mathematics, University of Pittsburgh, 301
Thackeray Hall, Pittsburgh, PA 15260, USA  \\
e-mail: \textit{hajlasz@pitt.edu}

\end{document}